\documentclass[11pt,letterpaper]{amsart}

\usepackage{amssymb, amsmath, amsfonts, latexsym}
\usepackage{enumerate}

\usepackage{color}
\usepackage{diagbox}
\newtheorem{thm}{Theorem}[]
\newtheorem{lem}{Lemma}[section]

\newtheorem{rmk}{Remark}[section]
\theoremstyle{definition}

\usepackage{graphicx}
\usepackage{subcaption}
\numberwithin{equation}{section} \theoremstyle{remark}

\title[Limiting Distribution of the Rightmost Eigenvalue ]{Limiting distributions of the rightmost eigenvalue of the truncated circular Unitary ensemble across all truncation regimes}

\author{Y\MakeLowercase{utao} M\MakeLowercase{a} and X\MakeLowercase{ujia} M\MakeLowercase{eng}}
\address{School of Mathematical Sciences $\&$ Laboratory  of Mathematics and Complex Systems of Ministry of Education, Beijing Normal University, 100875 Beijing, China.}
\thanks{The research of Yutao Ma was supported in part by NSFC 12571149 and 985 Projects.}
\email{mayt@bnu.edu.cn, 202321130122@mail.bnu.edu.cn}
\begin{document}

	\begin{abstract}
We study the limiting distribution of the largest real part of the eigenvalues of truncated unitary matrices. Let \(U_m\) be an \(m\times m\) Haar-distributed unitary matrix, and let \(Z_1,\ldots,Z_n\) be the eigenvalues of its upper-left \(n\times n\) principal submatrix. We determine the limiting distribution of \(\max_{1\le i\le n}\Re Z_i\) under a unified normalization: the limiting law is Gumbel when $q:=m-n\to\infty$ and is of Weibull type when $q$ is fixed. The fixed-\(q\) limiting laws converge to the Gumbel law as \(q\to\infty\). We also give a unified formulation of the normalizing constants for the spectral radius that recovers the previously known limiting laws. 	\end{abstract}
\maketitle

	\textbf{Keywords}:  Truncated unitary matrices; determinantal point processes; rightmost eigenvalue; Gumbel distribution; Weibull-type distribution.

	\textbf{Mathematics Subject Classification (2020):} 60B20, 60B10, 62G32.
\section{Introduction}
The eigenvalues of principal submatrices of Haar-distributed unitary matrices have been studied extensively in random matrix theory, in connection with quantum chaos and determinantal point processes on the unit disk. Zyczkowski and Sommers \cite{ZyczkowskiSommers2000} derived the joint eigenvalue density and established the limiting empirical spectral distribution. Subsequent work by Collins \cite{Collins2005}, Dong et al.\ \cite{Dong2012}, and Jiang \cite{Jiang2009} developed further connections with Jacobi ensembles, the circular law, and the arc law.

Let \(U_m\) be an \(m\times m\) Haar-distributed unitary matrix. We consider its upper-left \(n\times n\) principal submatrix and denote its eigenvalues by \(Z_1,\ldots,Z_{n}\). Write \(q:=m-n\) for the number of deleted rows and columns. Throughout, we assume that \(n\to\infty\) as \(m\to\infty\).
 By \cite{ZyczkowskiSommers2000}, their joint density is proportional to
 \begin{equation}\label{den}
 	\prod_{1\leq j<k\leq n}|z_{j}-z_{k}|^{2}\prod_{j=1}^{n}(1-|z_{j}|^{2})^{q-1}.
 \end{equation}
 The correlation kernel of this determinantal point process, given in \cite{Byun2025}, reads
 $$ 	K_{n}(z,w)=\frac{q(1-|w|^2)^{(q-1)/2}(1-|z|^2)^{(q-1)/2}}{\pi(1-w\bar{z})^{q+1}}\bigl(1-I_{w\bar{z}}(n,q+1)\bigr),$$
where \(I_x(a,b)\) denotes the regularized incomplete beta function
\[
I_x(a,b)=\frac{\Gamma(a+b)}{\Gamma(a)\Gamma(b)}\int_0^x u^{a-1}(1-u)^{b-1}\,du.
\]
The kernel admits a finite-rank representation in terms of orthonormal functions on the unit disk \(\mathbb{D}\). Indeed, the identity
\[
\sum_{j=0}^{n-1}\frac{q\,\Gamma(q+1+j)}{\Gamma(q+1)\Gamma(j+1)}\,t^j
=\frac{q}{(1-t)^{q+1}}\bigl(1-I_t(n,q+1)\bigr),
\]
yields
\[
K_{n}(z,w)=\sum_{j=0}^{n-1}\phi_j(w)\overline{\phi_j(z)},
\]
where the functions \(\phi_j\) are obtained by normalizing the monomials \(z^j\) with respect to the weight \((1-|z|^2)^{q-1}\) on \(\mathbb{D}\) and multiplying by the square root of this weight. Explicitly,
\begin{equation}\label{phiexp}
\phi_j(z)=\sqrt{\frac{\Gamma(q+1+j)}{\pi\,\Gamma(q)\Gamma(j+1)}}\,(1-|z|^2)^{\frac{q-1}{2}}z^j,\qquad z\in\mathbb{D},
\end{equation}
with normalization \(\int_{\mathbb{D}}|\phi_j|^2\,d^2z=1\).

The extreme value statistics of eigenvalues of such truncated unitary matrices have been studied extensively in recent years. For the spectral radius \(\max_{1\le j\le n}|z_j|\), Jiang and Qi \cite{JiangQi2017} proved convergence to the Gumbel distribution under the condition that \(n/m\) is bounded away from \(0\) and \(1\). Subsequently, Gui and Qi \cite{GuiQi2018} extended this result to other truncation regimes, establishing Gumbel convergence when \(n/m\to 0\), when \(q/(\log m)^3\to\infty\), and when  \(1\ll q\ll\log m\), and obtained a reversed Weibull limit when \(q\) is fixed. Miao and Qi \cite{MiaoQi2021} later filled the remaining gap when \(q\) lies between the scales \(\log m\) and \((\log m)^3\), thereby completing the weak convergence theory for the spectral radius.

These results use different normalizations in different ranges of the truncation parameter \(q\). A single formulation valid across all regimes is therefore desirable. In addition, the cited works concern the spectral radius; the corresponding problem for the largest real part has not been addressed for the truncated unitary ensemble. Related results are available for other non-Hermitian models, including the complex Ginibre ensemble and general i.i.d.\ complex random matrices \cite{Cipolloni2022, Cipolloni2023, HuMa2026, MaMeng2026} etc.

In this paper, we determine the limiting distribution of the largest real part for both $q\to\infty$ and fixed \(q\), using normalizing constants defined by a single implicit equation. The limiting law is Gumbel in the former case and of power-exponential type in the latter. The fixed-\(q\) laws converge to the Gumbel law as \(q\to\infty\). We also give a unified normalization for the spectral radius, recovering the limiting laws in \cite{GuiQi2018, JiangQi2017, MiaoQi2021}.

Recall that $q=m-n$. In all three regimes, let $\lambda_n\in(0,1)$
be the solution of
\begin{equation}\label{equation}\tag{U}
 \frac{mq}{n}(\lambda-1-\log\lambda)-\frac12\log\lambda
 +\frac52\log(1-\lambda)
 =\log\frac{n^{3/2}q^{q+1/2}}
 {2m(q+1)\sqrt\pi\,\Gamma(q+3/2)e^q}.
\end{equation}
We keep the gamma factor in exact form, including in the fixed $q$ regime.
The left-hand side is strictly decreasing, with derivative
\[
 -\frac{mq(1-\lambda)}{n\lambda}-\frac1{2\lambda}
 -\frac5{2(1-\lambda)}<0,
\]
and its limits at $0$ and $1$ are $+\infty$ and $-\infty$, respectively.
Thus, the solution exists and is unique. In Sections 3--5, we derive the asymptotic form of \(\lambda_n\) appropriate to each range of \(q\).
%Write$s_n=mn/q$.

Let $R_n=\max_{1\le i\le n}\Re Z_i$, and define
$$\widehat A_n=\sqrt{1-\frac{q\lambda_n}{m}},\qquad
\widehat B_n=\frac{\widehat A_n\lambda_n}{2m(1-\lambda_n)},
\qquad
X_n=\frac{R_n-\widehat A_n}{\widehat B_n}. $$
The following theorem gives the limiting distribution of \(X_n\).

\begin{thm}\label{main}
	Let \(X_n\) be defined as above. Then the following hold.
	\begin{itemize}
		\item If \(q\gg 1\), then for every fixed \(x\in\mathbb{R}\),
		\[
		\lim_{n\to\infty}\mathbb{P}(X_n\le x)=e^{-e^{-x}}.
		\]
		\item If \(q\) is a fixed positive integer, then
		\[
		\lim_{n\to\infty}\mathbb{P}(X_n\le x) =
		\begin{cases}
			e^{-(1-x/q)^{q+1/2}}, & x \leq q; \\
			1, & x > q.
		\end{cases}
		\]
	\end{itemize}
\end{thm}
\begin{rmk}
	The fixed-\(q\) limiting distribution functions converge pointwise to the Gumbel distribution function as \(q\to\infty\). Indeed, for each fixed \(x\in\mathbb{R}\),
	\[
	\lim_{q\to\infty}\bigl(1-\frac{x}{q}\bigr)^{q+1/2}=e^{-x}.
	\]
	It follows that
	\[
	\lim_{q\to\infty}\exp\bigl\{-\bigl(1-\frac{x}{q}\bigr)^{q+1/2}\bigr\}
	=
	e^{-e^{-x}}.
	\]
	This describes the transition from the fixed-\(q\) limiting laws to the Gumbel law.
\end{rmk}

\begin{rmk}
	Refining the trace estimates below gives the pointwise error bounds
	\[
	\begin{cases}
		(\log n)^{-1},
		& \log n\ll q<m,\\[1mm]
		q^{-1},
		& 1\ll q\lesssim\log n,\\[1mm]
		m^{-1/(2q+1)},
		& q \text{ fixed},
	\end{cases}
	\]
	for each fixed \(x\).
	The first two bounds are consistent at the crossover scale
	\(q\asymp\log n\), where \(q^{-1}\asymp(\log n)^{-1}\).
	For slowly diverging \(q\), the error bound deteriorates as \(q\)
	decreases, reflecting the slower approach to the Gumbel limit.
	When $q$ is fixed, the limiting law is of Weibull type rather than Gumbel, and the error scale is determined by
	\(\lambda_n\asymp m^{-1/(2q+1)}\).
\end{rmk}

\begin{rmk}
	For the spectral radius, the results of \cite{JiangQi2017,GuiQi2018,MiaoQi2021} use different normalizations in different regimes. Our method also yields a unified normalization for the spectral radius, recovering these limiting laws. Specifically, let \(\widetilde{\lambda}_n\in(0,1)\) be the unique solution of
	$$	\frac{mq}{n}(\widetilde{\lambda}-1-\log\widetilde{\lambda})+2\log(1-\widetilde{\lambda})
	=
	\log\frac{n^{3/2}q^q}{m^{1/2}\Gamma(q+2)e^{q}}. $$
	Define
	\[
	\widetilde{A}_n=\sqrt{1-\frac{q\widetilde{\lambda}_n}{m}},\qquad
	\widetilde{B}_n=\frac{\widetilde{A}_n\widetilde{\lambda}_n}{2m(1-\widetilde{\lambda}_n)},\qquad
	\widetilde{X}_n=\frac{\max_{1\le i\le n}|Z_i|-\widetilde{A}_n}{\widetilde{B}_n}.
	\]
	Then the following hold.
	\begin{itemize}
		\item If \(q\gg 1\), then for every fixed \(x\in\mathbb{R}\),
		\[
		\lim_{n\to\infty}\mathbb{P}(\widetilde{X}_n\le x)=e^{-e^{-x}}.
		\]
		\item If \(q\) is fixed, then
		\[
		\lim_{n\to\infty}\mathbb{P}(\widetilde{X}_n\le x)=
		\begin{cases}
			e^{-(1-x/q)^{q}},& x\le q,\\
			1,& x>q.
		\end{cases}
		\]
	\end{itemize}
	The proof follows the same strategy as that for the largest real part, with the half-plane exceedance region replaced by a radial exceedance region. The angular factor is consequently absent from the corresponding integrals, and we omit the analogous details.
\end{rmk}

\begin{rmk}
		The density \eqref{den} can be written as
	$$\prod_{1\leq j<k\leq n}|z_{j}-z_{k}|^{2}\prod_{j=1}^{n} \exp(-n \sum_{j=1}^n V_n(|z_j|))$$ with
	\[
	V_n(r)=-\frac{q-1}{n}\log(1-r^2).\] This is a Coulomb gas with radial $n$-dependent external fields.
	Set \[
	t_0=\sqrt{\frac{n}{m-1}},
	\qquad
	\alpha_n=\frac{4(m-1)^2}{n(q-1)}.
	\]
	Since $n\alpha_n t_0^2\sim4s_n$, where $s_n=mn/q$, we set, similarly to \cite{MaMeng2026v},
	\[
	\gamma_n^{\mathrm{gen}}
	=\frac12\log
	\frac{n\alpha_n t_0^2}{8\pi^4(\log s_n)^5},
	\qquad
	X_n^{\mathrm{gen}}
	=\sqrt{n\alpha_n\gamma_n^{\mathrm{gen}}}
	\biggl(
	R_n-t_0-\sqrt{\frac{\gamma_n^{\mathrm{gen}}}{n\alpha_n}}
	\biggr).
	\]
	For $(\log n)^3\ll q<m$, adapting the proof of \cite{MaMeng2026v}
	to $V_n$ gives
	\[
	\lim_{n\to\infty}\mathbb P(	X_n^{\mathrm{gen}}\le x)=\exp(-e^{-x}).
	\]
	In this regime, the explicit normalization above is
	asymptotically equivalent to the implicit normalization \eqref{equation}.
	
	Under the stronger condition $(\log n)^5\lesssim q<m$,
	the leading correction introduced by this explicit
	normalization is of order
	$\log\log s_n/\log s_n\asymp\log\log n/\log n$.
	It arises from approximating the prefactor
	$(\log s_n/(2\gamma_n^{\mathrm{gen}}))^{5/4}$ by $1$.
\end{rmk}

Throughout the paper, we use standard asymptotic notation. For a positive sequence $(b_n)$ and a real sequence $(a_n)$, we write
\[
a_n = O(b_n) \quad\text{if } \exists\, C>0 \text{ such that } |a_n| \le C b_n \text{ for all sufficiently large } n,
\]
and
\[
a_n = o(b_n) \quad\text{if } \lim_{n\to\infty} \frac{a_n}{b_n} = 0.
\]
We write \(a_n \ll b_n\) to mean \(a_n = o(b_n)\), \(a_n \lesssim b_n\) to mean \(a_n = O(b_n)\), and $a_n\asymp b_n$ to mean both $a_n\lesssim b_n$ and $b_n\lesssim a_n$. We write \(a_n \sim b_n\) when \(a_n/b_n \to 1\). All asymptotic relations are understood as \(n\to\infty\), unless otherwise specified.

The remainder of the paper is organized as follows. Section~2 introduces the squared statistic, reduces the proof to it, and collects the preliminary estimates for the matrix entries. Sections 3, 4 and 5 prove Theorem~1 in the regimes \(m^{2/5}(\log m)^2\lesssim q<m\), \(1\ll q\ll m^{2/5}(\log m)^2\), and fixed \(q\), respectively.

\section{Reduction to the squared statistic and preliminary estimates}

In this section we introduce the squared statistic \(\widehat X_n\),
reduce the linear statistic \(X_n\) to it, outline the proof of
Theorem~1, and collect the preliminary estimates for the matrix
entries \(M(x)\) used in the subsequent sections.

\subsection{Reduction to the squared statistic}

For the subsequent proofs, it is convenient to introduce
\[
\widehat X_n=\frac{R_n^2-A_n}{B_n},
\]
where
$$A_n=\widehat A_n^2=1-\frac{q\lambda_n}{m},\qquad
B_n=2\widehat A_n\widehat B_n
=\frac{A_n\lambda_n}{m(1-\lambda_n)}. $$

The two normalized statistics have the same limiting distribution.
To see this, note that the estimates in Sections 3--5 give
$B_n/A_n\to0$. If $\widehat X_n$ converges in distribution,
its tightness implies
\[
\frac{R_n^2}{A_n}
=1+\frac{B_n}{A_n}\widehat X_n
\xrightarrow{\mathbb P}1.
\]
Moreover, rotational symmetry and the determinantal gap formula
give $\mathbb P(R_n<0)\le e^{-n/2}$, as shown below.
Hence $R_n/\sqrt{A_n}\to1$ in probability.
On the event $\{R_n\ge0\}$, rationalizing the numerator yields
\[
X_n
=\frac{2\sqrt{A_n}}{R_n+\sqrt{A_n}}\,\widehat X_n.
\]
The prefactor converges to one in probability, while the
complementary event has probability tending to zero.
Slutsky's theorem therefore shows that $X_n$ converges to
the same limit as $\widehat X_n$.
It thus suffices to analyze $\widehat X_n$.

\subsection{Outline of the proof}
We now outline the proof for the largest real part.
Write $L_n(x)=A_n+B_nx$. For $L_n(x)>0$, define
\[
A(x)=\{z\in\mathbb D:\Re z>\sqrt{L_n(x)}\}.
\]
Then the event $\{R_n\le\sqrt{L_n(x)}\}$ is precisely
the event that $A(x)$ contains no eigenvalues.

The right half-disk \(\{z\in\mathbb{D}:\ \Re z>0\}\) has expected eigenvalue count \(n/2\) by rotational symmetry. Since the eigenvalue process is determinantal with a trace-class kernel \(\mathbb{K}_n\), its gap probability satisfies
\[
\mathbb{P}\bigl(R_n<0\bigr)
=\det({\rm I}-\mathbb{K}_n|_{\Re z>0})\le e^{-\operatorname{Tr} (\mathbb{K}_n|_{\Re z>0})}=e^{-n/2}.
\]

The event \(\{\widehat X_n\le x\}\) differs from  \(\{R_n\le \sqrt{L_n(x)}\}\) only when \(R_n<-\sqrt{L_n(x)}\). Since this event is contained in \(\{R_n<0\}\), the bound above yields
\[
\left|
\mathbb P(\widehat X_n\le x)
-\mathbb P\bigl(R_n\le \sqrt{L_n(x)}\bigr)
\right|
\le
\mathbb P\bigl(R_n<-\sqrt{L_n(x)}\bigr)
\le
\mathbb P(R_n<0)
\le e^{-n/2}.
\]

Thus, the distribution function of the squared statistic differs from the corresponding gap probability by an exponentially small error. The determinantal gap probability formula gives
\[
\mathbb{P}(R_n\le\sqrt{L_n(x)})=\det(\rm I-\mathbb{K}_n|_{A(x)}).
\]
Since the operator \(\mathbb{K}_n\) corresponding to $K_n$ has finite rank, this Fredholm determinant reduces to a matrix determinant. Define the \(n\times n\) matrix \(M(x)=(M_{j,k}(x))_{1\le j,k\le n}\) by
\begin{equation}\label{MJkdefine}
	M_{j, k}(x)=\int_{A(x)}\overline{\phi_{n-j}(z)}\,\phi_{n-k}(z)\,d^2z,
\end{equation}
where $\phi_j$ is given by \[
	\phi_j(z)=\sqrt{\frac{\Gamma(q+1+j)}{\pi\,\Gamma(q)\Gamma(j+1)}}\,(1-|z|^2)^{\frac{q-1}{2}}z^j,\qquad z\in\mathbb{D}.
	\]  
The finite-rank structure then yields the identity
\[
\det(\rm I-\mathbb{K}_n|_{A(x)})=\det(I_{n}-M(x)),
\]
which gives the probability representation
\begin{equation}\label{prob}
	\mathbb{P}(\widehat{X}_n\le x)=\det({\rm I}_{n}-M(x))+O(e^{-n/2}).
\end{equation}

We estimate the determinant in \eqref{prob} using the following standard bound for Fredholm determinants (see, e.g., \cite{Gohberg}):
\begin{equation}\label{trhs}
	\Bigl|\det(\rm I-\mathbb{K}_n|_A)-\exp\bigl(-\operatorname{Tr}(\mathbb{K}_n|_A)\bigr)\Bigr|
	\le \|\mathbb{K}_n|_A\|_2
	\exp\Bigl\{\frac{1}{2}\bigl(\|\mathbb{K}_n|_A\|_2+1\bigr)^2-\operatorname{Tr}(\mathbb{K}_n|_A)\Bigr\}.
\end{equation}
Consequently, whenever the Hilbert--Schmidt norm \(\|\mathbb{K}_n|_A\|_2\) is \(o(1)\), we have the asymptotic equivalence
\[
\det(\rm I-\mathbb{K}_n|_A)\sim \exp\bigl(-\operatorname{Tr}(\mathbb{K}_n|_A)\bigr).
\]
We verify this condition for \(A=A(x)\).
The definition of \(M(x)\) indicates that
\[
\operatorname{Tr}\bigl(\mathbb{K}_n|_{A(x)}\bigr)=\operatorname{Tr}(M(x)),
\qquad
\bigl\|\mathbb{K}_n|_{A(x)}\bigr\|_2^2=\|M(x)\|_{\mathrm{HS}}^2.
\]
Thus, it remains to establish
	\[
	\operatorname{Tr} (M(x))\longrightarrow
	\begin{cases}
	e^{-x}, & q\to\infty,\quad x\in\mathbb R,\\[2mm]
	(1-x/q)^{q+1/2}, & q\text{ fixed},\quad x<q,
	\end{cases}
	\quad\text{and}\quad
	\|M(x)\|_{\mathrm{HS}}=o(1).
	\]
	
\subsection{Preliminary estimates for the matrix entries}

	We begin with integral representations for the entries of \(M(x)\), which will be used throughout the proofs. 	
	The following lemma provides exact integral representations for these entries.

\begin{lem}\label{integral-rep}
Let $L=L_n(x)=A_n+B_nx$, with $0<L<1$, and set
\[
C_t=\sqrt{\frac{\Gamma(m-t+1)}
{\pi\Gamma(q)\Gamma(n-t+1)}},
\qquad
\beta_L(u)=\arccos\sqrt{L/u}
\]
for $t\in [1, n]$ and $u\ge L.$
For $1\le j\le n$, we have
\begin{align}
M_{j,j}(x)
&=C_j^2\int_L^1(1-u)^{q-1}u^{n-j}\beta_L(u)\,du,
\label{wMjj1}\\
&=\frac{C_j^2}{2}\int_L^1\frac{ds}{\sqrt{s(1-s)}}
\int_{L/s}^1(1-u)^{q-1}u^{n-j}\,du.
\label{wMjj2}
\end{align}
Moreover, $M_{j,j}(x)$ is nonincreasing in $j$.
For $j\ne k$, we have
\begin{equation}\label{wMjk}
|M_{j,k}(x)|
\le \frac{C_{(j+k)/2}^2}{|k-j|}
\int_L^1(1-u)^{q-1}u^{n-(j+k)/2}\,du.
\end{equation}

\end{lem}

\begin{proof}
Write $z=re^{i\theta}$ and $u=r^2$. For fixed $u\in(L,1)$,
the angular interval is $(-\beta_L(u),\beta_L(u))$.
Integrating over this interval and using $r\,dr=du/2$
gives \eqref{wMjj1}. The identity
\[
\beta_L(u)=\frac12\int_{L/u}^1\frac{ds}{\sqrt{s(1-s)}}
\]
and Tonelli's theorem then yield \eqref{wMjj2}.

To prove the monotonicity, let
$Y_{n-j}\sim\operatorname{Beta}(n-j+1,q)$.
By \eqref{wMjj2},
\[
M_{j,j}(x)
=\frac1{2\pi}\int_L^1
\frac{\mathbb P(Y_{n-j}>L/s)}{\sqrt{s(1-s)}}\,ds.
\]
For fixed second shape parameter $q$, the beta distribution is stochastically increasing in its first shape parameter. Hence the
integrand, and therefore $M_{j,j}(x)$, is nonincreasing in $j$.

For $j\ne k$, the angular integral satisfies
\[
\biggl|
\int_{-\beta_L(u)}^{\beta_L(u)}
e^{i(j-k)\theta}\,d\theta
\biggr|
=
\frac{2|\sin((j-k)\beta_L(u))|}{|j-k|}
\le \frac{2}{|j-k|}.
\]
Consequently,
\[
|M_{j,k}(x)|
\le \frac{C_jC_k}{|j-k|}
\int_L^1(1-u)^{q-1}u^{n-(j+k)/2}\,du.
\]
Since
\[
C_t^2=\frac1{\pi B(n-t+1,q)}
\]
and $\log B(a,q)$ is convex in $a>0$, we have
$C_jC_k\le C_{(j+k)/2}^2$.
This proves \eqref{wMjk}.
\end{proof}

	The following lemma provides a useful bound for the
Hilbert--Schmidt norm of $M(x)$.

	\begin{lem}\label{HS}
		Let \(\alpha_n\) and \(\beta_n\) be positive integer sequences satisfying \(\alpha_n\gg \beta_n\gg 1\). Assume that
		\[
		\operatorname{Tr}(M(x))=O(1),\qquad nM_{\alpha_n,\alpha_n}(x)\ll 1.
		\]
		Then, 
		\[
		\|M(x)\|_{\mathrm{HS}}^2
		\lesssim
		nM_{\alpha_n,\alpha_n}(x)
		+\beta_n\alpha_n\,(M_{1,1}(x))^2
		+\sum_{\substack{1\le j< k\le \alpha_n\\ k-j>\beta_n}}(M_{j,k}(x))^2.
		\]
	\end{lem}

	\begin{proof}
		By definition,
		\[
		\|M(x)\|_{\mathrm{HS}}^2=\sum_{1\le j,k\le n}(M_{j,k}(x))^2.
		\]
		We split the double sum into three regions.

		By the Cauchy--Schwarz inequality, the entries of \(M(x)\) satisfy
		\begin{equation}\label{CSIneq}
			M_{j,k}^2(x)\le M_{j,j}(x)M_{k,k}(x).
		\end{equation}
		First, for \(j,k>\alpha_n\), the monotonicity of \(M_{j,j}\) in \(j\) together with \eqref{CSIneq} gives
		\[
		\sum_{\alpha_n+1\le j,k\le n}(M_{j,k}(x))^2
		\le n^2(M_{\alpha_n+1,\alpha_n+1}(x))^2\le n^2(M_{\alpha_n,\alpha_n}(x))^2.
		\]
		Second, by symmetry, the regions $k\le\alpha_n<j$ and
$j\le\alpha_n<k$ contribute equally, so it suffices to
consider the former. By \eqref{CSIneq},
\[
\sum_{1\le k\le\alpha_n<j\le n}(M_{j,k}(x))^2
\le
\sum_{j=\alpha_n+1}^{n}M_{j,j}(x)
\sum_{k=1}^{n}M_{k,k}(x)
\le
nM_{\alpha_n,\alpha_n}(x)\operatorname{Tr}(M(x)).
\]
		Since \(\operatorname{Tr}(M(x))=O(1)\), this contribution is bounded by \(nM_{\alpha_n,\alpha_n}(x)\).

	Finally, consider $1\le j\le k\le\alpha_n$, again using symmetry. If \(k-j\le \beta_n\), the number of such pairs is \(O(\beta_n\alpha_n)\), and by Cauchy--Schwarz and the monotonicity of the diagonal entries,
		\[
		(M_{j,k}(x))^2\le M_{j,j}(x)M_{k,k}(x)\le (M_{1,1}(x))^2.
		\]
		Hence
		\[
		\sum_{\substack{1\le j\le k\le \alpha_n\\ k-j\le \beta_n}}(M_{j,k}(x))^2
		\lesssim \beta_n\alpha_n\,(M_{1,1}(x))^2.
		\]
		The terms with \(k-j>\beta_n\) give the final sum in the statement. Combining these estimates proves the lemma.
	\end{proof}
	The following asymptotic will be used later in the estimation of \(M_{j,k}(x)\).
	
	\begin{lem}\label{int}
		Given \(0<v<1\). We have
		\[
		\int_{v}^{1}
		\frac{(\frac{a}{s}-b)^{-1}}{\sqrt{s(1-s)}}
		\exp\{-\frac{1}{2}(\frac{a}{s}-b)^2\}\,ds
		\sim
		\frac{\sqrt{\pi}}{\sqrt{a}(a-b)^{\frac32}}
		\exp\{-\frac{1}{2}(a-b)^2\},
		\]
		provided that \(a,b\gg 1\), \(a-b\gg 1\), and \(a(1-v)/v\gg 1\).
	\end{lem}
	\begin{proof}
		Make the change of variables
		\[
		u:=\frac{a}{s}-b.
		\]
		Then \(s=a/(u+b)\), and hence
		\[
		\begin{aligned}
			I:&=\int_{v}^{1}
			\frac{(\frac{a}{s}-b)^{-1}}{\sqrt{s(1-s)}}
			\exp\bigl\{-\frac12\bigl(\tfrac{a}{s}-b\bigr)^2\bigr\}\,ds\\
			&=\sqrt{a}\int_{a-b}^{\frac{a}{v}-b}
			\frac{1}{\sqrt{u+b-a}\,(u+b)u}
			e^{-u^2/2}\,du.
		\end{aligned}
		\]
		Define \(c=a-b\) and \(t=u-c\). It follows that
		\[
		I=\sqrt{a}\,e^{-c^2/2}
		\int_{0}^{\frac{a}{v}-b-c}
		\frac{1}{\sqrt{t}\,(t+a)(t+c)}
		e^{-t^2/2-ct}\,dt.
		\]
		Note that
		\[
		\frac{a}{v}-b-c
		=a\Bigl(\frac{1-v}{v}\Bigr)\gg 1,
		\]
		which is also much larger than \(c^{-2/3}\). We split the integral at \(t=c^{-2/3}\).
		
		For \(0<t<c^{-2/3}\), we have \(t^2< c^{-4/3}\ll 1\), and
		\[
		(t+a)(t+c)=ac(1+o(1)).
		\]
		Thus, with the substitution \(y=ct\),
		\[
		\begin{aligned}
			\int_{0}^{c^{-2/3}}
			\frac{1}{\sqrt{t}(t+a)(t+c)}
			e^{-t^2/2-ct}\,dt
			&=\frac{1+o(1)}{ac}
			\int_{0}^{c^{-2/3}}\frac{e^{-ct}}{\sqrt{t}}\,dt\\
			&=\frac{1+o(1)}{ac^{3/2}}
			\int_{0}^{c^{1/3}}\frac{e^{-y}}{\sqrt{y}}\,dy\\
			&=\frac{\sqrt{\pi}}{ac^{3/2}}(1+o(1)).
		\end{aligned}
		\]
		For the tail \(t>c^{-2/3}\), the monotonicity of the integrand gives
		\[
		\int_{c^{-2/3}}^{a(1-v)/v}
		\frac{1}{\sqrt{t}(t+a)(t+c)}
		e^{-t^2/2-ct}\,dt
		\lesssim
		\frac{1}{ac^{2/3}}
		\int_{c^{-2/3}}^{\infty}e^{-ct}\,dt
		\lesssim
		\frac{e^{-c^{1/3}}}{ac^{5/3}},
		\]
		which is negligible compared with the leading term \(\sqrt{\pi}/(ac^{3/2})\) since $c=a-b\gg 1$. Combining the two estimates yields
		\[
		I\sim\frac{\sqrt{\pi}}{\sqrt{a}(a-b)^{3/2}}
		\exp\{-\frac12(a-b)^2\},
		\]
		which finishes the lemma.
	\end{proof}

Next, we give some understanding on the unique solution $\lambda_n$ to the equation \eqref{equation}:  
$$ \frac{mq}{n}(\lambda-1-\log\lambda)-\frac12\log\lambda
 +\frac52\log(1-\lambda)
 =\log\frac{n^{3/2}q^{q+1/2}}
 {2m(q+1)\sqrt\pi\,\Gamma(q+3/2)e^q}.$$ 
 We first consider the case $$
1\ll q\ll m^{2/5}(\log m)^2.$$  

\begin{lem}\label{nlambda-lower}\label{lambda-bound}
Under the assumptions $
1\ll q\ll m^{2/5}(\log m)^2$, we have
\begin{equation}\label{eq:result}
m\lambda_n\gg m^{9/10},
\qquad
1\ll\sqrt q(1-\lambda_n)\lesssim\sqrt{\log m}.
\end{equation}
Moreover, the equation \eqref{equation}  can be written as
\begin{equation}\label{limequation}
q(\lambda_n-1-\log\lambda_n)
-\frac12\log\lambda_n
+\frac52\log(1-\lambda_n)
=
\frac12\log\frac{m}{8\pi^2q^3}+o(1).
\end{equation}
\end{lem}

\begin{proof}
Since $n\sim m$, Stirling's formula applied to
\eqref{equation} gives
\[
\frac{mq}{n}(\lambda_n-1-\log\lambda_n)
-\frac12\log\lambda_n
+\frac52\log(1-\lambda_n)
=
\frac12\log\frac{m}{8\pi^2q^3}+o(1).
\]
At $\lambda=1-q^{-1/2}$, the left-hand side equals
$-\frac54\log q+O(1)$, which is smaller than the
right-hand side for all sufficiently large $n$.
By monotonicity, $1-\lambda_n>q^{-1/2}$. Consequently,
\[
\frac{mq}{n}(\lambda_n-1-\log\lambda_n)
\le
\frac12\log m-\frac14\log q+O(1)
=O(\log m).
\]
Using
\[
\lambda-1-\log\lambda\ge\frac12(1-\lambda)^2,
\qquad
-\log\lambda\le\lambda-1-\log\lambda+1,
\]
we obtain
\[
q(1-\lambda_n)^2=O(\log m),
\qquad
-\log\lambda_n=O(1+\frac{\log m}{q})=o(\log m).
\]
The latter implies
$m\lambda_n=m^{1-o(1)}\gg m^{9/10}$.

It remains to show that $\sqrt q(1-\lambda_n)\to\infty$.
If this quantity were bounded along a subsequence,
expansion of the defining equation would give
\[
\frac12q(1-\lambda_n)^2
+\frac52\log\bigl(\sqrt q(1-\lambda_n)\bigr)
=
\frac12\log m-\frac14\log q+O(1).
\]
The left-hand side is bounded above, whereas the
right-hand side tends to infinity. This contradiction
proves \eqref{eq:result}.

Finally, since $m-n=q$ and
$q(\lambda_n-1-\log\lambda_n)=O(\log m)$, we have
\[
\bigl(\frac{mq}{n}-q\bigr)
(\lambda_n-1-\log\lambda_n)
=
O\bigl(\frac{q\log m}{n}\bigr)
=o(1).
\]
Thus replacing $mq/n$ by $q$ in the defining equation
introduces only an $o(1)$ error, proving
\eqref{limequation}. The proof is completed. 
\end{proof}

Now we discuss the regime \(m^{2/5}(\log m)^2\lesssim q<m\). 

Set
\[
\gamma_n=\frac{mq}{n}(1-\lambda_n)^2,
\qquad
s_n=\frac{mn}{q}
\]
such that 
\[
\lambda_n=1-\sqrt{\frac{n\gamma_n}{mq}}.
\]
We start with the asymptotic of \(\gamma_n\).

\begin{lem}\label{unified-gaussian}
Let $s_n$ and $\gamma_n$ be defined as above. For the case \(m^{2/5}(\log m)^2\lesssim q<m,\) we have
\begin{equation}\label{gamma-order}
    \gamma_n\sim\frac12\log s_n,
\end{equation}
and 
\begin{equation}\label{gamma-match}
    \frac{s_n^{1/4}e^{-\gamma_n/2}}
     {2\sqrt2\pi\gamma_n^{5/4}}
=1+o(\gamma_n^{-1}).
\end{equation}
\end{lem}

\begin{proof}
Let $F_n(t)$ denote the left-hand side of \eqref{equation}
minus its right-hand side after the substitution
\[
\lambda=1-\sqrt{\frac{nt}{mq}},
\qquad 0<t<\frac{mq}{n}.
\]
By the definition of $\gamma_n$, the unique solution $\lambda_n$ to the equation \eqref{equation} 
corresponds exactly to $t=\gamma_n$. Thus
\[
F_n(\gamma_n)=0.
\]
Moreover, $F_n$ is strictly increasing, so $\gamma_n$ is its
unique zero.

Applying Stirling's formula to $\Gamma(q+\frac32),$ since $1\ll q$, the right-hand side of
\eqref{equation} equals
\begin{equation}\label{U-stirling}
\frac12\log\frac{n^3}{8m^2q^3\pi^2}
-\frac{35}{24q}+O(q^{-2}).
\end{equation}
For
\[
\frac14\log s_n\le t\le\log s_n,
\]
we have
\[
\sqrt{\frac{nt}{mq}}
\le\sqrt{\frac{\log s_n}{q}}
=o(1).
\]
Taylor expansion of the left-hand side of
\eqref{equation} therefore gives
\[
\frac t2+\frac54\log\frac{nt}{mq}
+O\Bigl(
\frac{t^{3/2}\sqrt n}{\sqrt{mq}}
+\sqrt{\frac{tn}{mq}}
\Bigr),
\]
uniformly on this interval. Combining this with
\eqref{U-stirling} and using $s_n=mn/q$, we obtain
\[
2F_n(t)
=t+\frac52\log t-\frac12\log s_n+\log(8\pi^2)
+O\Bigl(
\frac{t^{3/2}\sqrt n}{\sqrt{mq}}
+\sqrt{\frac{tn}{mq}}+\frac1q
\Bigr).
\]
This expansion holds throughout the stated interval;
we have not yet used the identity $F_n(\gamma_n)=0$.

Since $mq/n\ge q$ and $s_n\le m^2$, the assumptions
of this section imply
\[
\frac{(\log s_n)^{3/2}}{\sqrt q}+\frac1q
=o((\log s_n)^{-1}).
\]
In particular, the error in the expansion of $2F_n(t)$
is uniformly $o(1)$. At the two endpoints,
\[
2F_n\Bigl(\frac14\log s_n\Bigr)
=-\frac14\log s_n+O(\log\log s_n)<0
\]
and
\[
2F_n(\log s_n)
=\frac12\log s_n+O(\log\log s_n)>0
\]
for all sufficiently large $n$.
Since $F_n$ is strictly increasing and
$F_n(\gamma_n)=0$, it follows that
\[
\frac14\log s_n<\gamma_n<\log s_n.
\]

The expansion above thus applies at $t=\gamma_n$.
Using $F_n(\gamma_n)=0$ and the sharper error bound,
we obtain
\[
\gamma_n+\frac52\log\gamma_n
=\frac12\log s_n-\log(8\pi^2)
+o((\log s_n)^{-1}).
\]
The bounds on $\gamma_n$ imply
$\log\gamma_n=O(\log\log s_n)$, and hence
\[
\gamma_n=\frac12\log s_n+O(\log\log s_n).
\]
This proves \eqref{gamma-order}.

Finally, rearranging the same identity and using
$\gamma_n\asymp\log s_n$, we find
\[
\begin{aligned}
\log\Bigl(
\frac{s_n^{1/4}e^{-\gamma_n/2}}
{2\sqrt2\pi\gamma_n^{5/4}}
\Bigr)
=\frac14\log s_n-\frac{\gamma_n}{2}
-\frac54\log\gamma_n-\frac12\log(8\pi^2)=o(\gamma_n^{-1}).
\end{aligned}\]
Exponentiating yields \eqref{gamma-match}.
\end{proof} 
The following lemma provides a bound used to control the error terms
in the subsequent estimates.

\begin{lem}\label{snmnq}
	Under the assumption \(m^{2/5}(\log m)^2\lesssim q<m,\) uniformly for
	\(0\le a\le 1/8\), we have
	\[
	s_n^a\sqrt{\frac{m}{nq}}=o(1).
	\]
\end{lem}

\begin{proof}
	Recall that \(s_n=mn/q\) and \(m=n+q\). Thus
	\[
	s_n^a\sqrt{\frac{m}{nq}}
	=
	\frac{m^{a+1/2}}{n^{1/2-a}q^{a+1/2}}.
	\]
	If \(q\ge m/2\), then
	\[
	s_n^a\sqrt{\frac{m}{nq}}
	\lesssim n^{a-1/2}
	\le n^{-3/8}
	=o(1).
	\]
	If \(q<m/2\), then \(n>m/2\). Using
	\(q\gtrsim m^{2/5}(\log m)^2\), we obtain
	\[
	s_n^a\sqrt{\frac{m}{nq}}
	\lesssim m^{2a}q^{-a-1/2}
	\lesssim
	m^{(8a-1)/5}(\log m)^{-2a-1}
	\le (\log m)^{-1}
	=o(1).
	\]
\end{proof}

Next, we continue to present a lemma when 
\(m^{2/5}(\log m)^2\lesssim q<m.\) 

\begin{lem}\label{T1-T2}
	For fixed $x\in\mathbb R$ and $1\le j\le i_n:=\lfloor s_n^{8/15}\rfloor$, set
	$$ T_1=\frac{(m-j)^{3/2}L_n(x)}{\sqrt{(n-j)q}},
	\qquad
	T_2=\sqrt{\frac{(m-j)(n-j)}q}$$
and $$u_n(j, x)=\frac{j-1}{\sqrt{s_n}}+\sqrt{\gamma_n}+\frac{x}{\sqrt{\gamma_n}}.$$
When 
\(m^{2/5}(\log m)^2\lesssim q<m,\) we have 
$$
T_1\sim\sqrt{s_n},
\qquad
T_1-T_2\sim u_n(j,x),
\qquad
(T_1-T_2)^2-u_n^2(j,x)=o(1).
$$
%In particular,
%$$
%(T_1-T_2)^3\sqrt{\frac{m}{nq}}
%=o(1).
%$$
\end{lem} 

\begin{proof} We know from Lemma \ref{unified-gaussian} that the normalizing constants satisfy
\begin{equation}\label{An}
	A_n=1-\frac{q\lambda_n}{m}
	=\frac{n}{m}+\frac{\sqrt{n q\,\gamma_n}}{m^{3/2}}
	=\frac{n}{m}\bigl(1+O\bigl(\tfrac{\sqrt{\log s_n}}{\sqrt{s_n}}\bigr)\bigr),
\end{equation}
and
\begin{equation}\label{Bn}
	B_n=\frac{A_n\lambda_n}{m(1-\lambda_n)}
	=\frac{\sqrt{n q}}{m^{3/2}\sqrt{\gamma_n}}
	+O(m^{-1}\sqrt{\log s_n}).
\end{equation} 
%Since $\frac{i_n}{n}\asymp \frac{m^{8/15}}{n^{7/15}q^{8/15}},$ when $q\asymp m$ we see 
%$\frac{i_n}{n}\asymp n^{-7/15}$ and meanwhile if $m^2(\log m)^2\lesssim q\ll m,$ we know $n\sim m$ and then  
%\[
%\frac{i_n}{n}\asymp \frac{m^{1/15}}{q^{8/15}}\ll \frac{m^{1/15}}{m^{16/75}}.
%\]
%Thus, we always have 
%\begin{equation}\label{ino1}
%\frac{i_n}{n}\ll 1.	
%\end{equation}
By Lemma \ref{snmnq} with \(a=1/30\),
$$	\frac{i_n}{n}
\le \frac{s_n^{8/15}}{n}
=s_n^{1/30}\sqrt{\frac{m}{nq}}
=o(1). $$
Then, 
$$	L_n(x)=A_n+B_n x\sim \frac{n}{m} $$
 and consequently 
$$T_1\sim \frac{m^{3/2}}{\sqrt{nq}}\times \frac{n}{m}\sim \frac{\sqrt{mn}}{\sqrt{q}}=\sqrt{s_n}$$ uniformly on 
$1\le j\le i_n.$  

It holds 
\[
\begin{aligned}
T_1-T_2
&=\frac{(m-j)^{3/2}}{\sqrt{(n-j)q}}
\Bigl(L_n(x)-\frac{n-j}{m-j}\Bigr)\\
&=\frac{j}{\sqrt{s_n}}
\sqrt{\frac{1-j/m}{1-j/n}}
+\frac{(1-j/m)^{3/2}}{\sqrt{1-j/n}}
\bigl(
\sqrt{\gamma_n}
+\frac{x}{\sqrt{\gamma_n}}
\frac{mA_n\lambda_n}{n}
\bigr).
\end{aligned}
\] 
Since $j\le i_n=o(n)$, both factors involving $j/m$ and $j/n$
equal $1+O(i_n/n)$. Consequently,
\[
T_1-T_2
=u_n(j,x)
+O\bigl(u_n(j,x)\frac{i_n}{n}\bigr)
+O_x\bigl(
s_n^{-1/2}+\sqrt{\frac{n}{mq}}
\bigr).
\]
This verifies $T_1-T_2\sim u_n(j,x)$. More precisely,
\[
\begin{aligned}
|(T_1-T_2)^2-u_n^2(j,x)|
&\lesssim
u_n^2(j,x)\frac{i_n}{n}
+u_n(j,x)
\bigl(s_n^{-1/2}+\sqrt{\frac{n}{mq}}\bigr)\\
&\lesssim
\frac{s_n^{3/5}}{n}
+s_n^{-7/15}
+s_n^{1/30}\sqrt{\frac{m}{nq}}.
\end{aligned}
\]
%By definition, 
%$$\frac{s_n^{3/5}}{n}=\frac{m^{3/5}}{n^{2/5}q^{3/5}}=s_n^{1/10}\sqrt{\frac{m}{nq}},$$
%whence 
%$$|(T_1-T_2)^2-u_n^2(j,x)|\le \frac{m^{3/5}}{n^{2/5}q^{3/5}}+s_n^{-7/15}=o(1).$$ 
%Here, we use the same  argument as for \eqref{ino1} to get $\frac{m^{3/5}}{n^{2/5}q^{3/5}}=o(1).$
%Thus 
% $$(T_1-T_2)^2-u_n^2(j,x)=o(1).$$
%In particular,
%\[
%(T_1-T_2)^3\sqrt{\frac{m}{nq}}
%\lesssim \frac{i_n^3}{s_n^{3/2}}\sqrt{\frac{m}{nq}}\lesssim  \sqrt{\frac{m}{nq}}s_n^{1/10}=\frac{m^{3/5}}{n^{2/5}q^{3/5}}=o(1).
%\]
%The proof is completed. 
Since
\[
\frac{s_n^{3/5}}{n}
=s_n^{1/10}\sqrt{\frac{m}{nq}},
\]
Lemma \ref{snmnq}, applied with \(a=1/10\) and \(a=1/30\), gives
\[
(T_1-T_2)^2-u_n^2(j,x)=o(1)
\]
uniformly for \(1\le j\le i_n\). The proof is completed now. 
%Finally, \(T_1-T_2\asymp u_n(j,x)\lesssim s_n^{1/30}\), so
%\[
%(T_1-T_2)^3\sqrt{\frac{m}{nq}}
%\lesssim s_n^{1/10}\sqrt{\frac{m}{nq}}
%=o(1),
%\]
%again by Lemma \ref{snmnq}.
\end{proof}

We finish this section with the following estimates, which are borrowed from \cite{MaMeng2026}.

\begin{lem}\label{sum1}
	The following two estimates hold.
	\begin{enumerate}
		\item[(i)] For any fixed \(r\in\mathbb{R}\) and \(z\gg 1\),
		\[
		\int_{z}^{+\infty}t^{r} e^{-t^2/2}\,dt
		=
		z^{r-1}e^{-z^2/2}
		(1+\frac{r-1}{z^2}+O(z^{-4})).
		\]
		
		\item[(ii)] 
		Let $$u_n(j,x)=\frac{j-1}{\sqrt{s_n}}+\sqrt{\gamma_n}+\frac{x}{\sqrt{\gamma_n}}.$$ For any \(L\ge 1\) and \(x\in\mathbb{R}\) such that
		\[
		1\ll u_n(L,x)\ll s_n^{1/2},
		\]
		and any \(c>0\), we have
		\[
		\sum_{j=L}^{+\infty}u_n^{-c}(j,x)e^{-u_n^2(j,x)/2}
		\sim
		\frac{\sqrt{s_n}}{u_n^{c+1}(L,x)}e^{-u_n^2(L,x)/2}.
		\]
	\end{enumerate}
\end{lem}

\section{The regime \(m^{2/5}(\log m)^2\lesssim q<m\)}

Remember our task is on both ${\rm Tr}(M(x))$ and $\|M(x)\|_{\rm HS}.$ We prepare precise asymptotic for $M_{j, j}(x)$ and proper upper bounds for $M_{j, k}(x)$ for $j\neq k.$  

\begin{lem}\label{beta-tail}
Let \(i_n\) and \(u_n(j,x)\) be as above, and let \(\tau(s)=\frac{T_1}{s}-T_2\).
Suppose \(Y_{n-j}\) is a random variable obeying the distribution \(\operatorname{Beta}(n-j+1,q)\).
Then
\begin{equation}\label{beta-tail-local}
\mathbb P\Bigl(Y_{n-j}>\frac{L_n(x)}s\Bigr)
=
\frac{1+o(1)}{\sqrt{2\pi}\,\tau(s)}
e^{-\tau^2(s)/2}
\end{equation}
uniformly for $1\le j\le i_n$ and
$L_n(x)\le s\le1$ such that $\tau(s)\le2\tau(1)$.
For $\tau(s)>2\tau(1)$,
\begin{equation}\label{beta-tail-far}
\mathbb P\Bigl(Y_{n-j}>\frac{L_n(x)}s\Bigr)
\lesssim \frac{e^{-2\tau^2(1)}}{\tau(1)}.
\end{equation}
In particular,
\begin{equation}\label{Pv}
\mathbb P(Y_{n-j}>L_n(x))
=
\frac{1+o(1)}{\sqrt{2\pi}\,u_n(j,x)}
e^{-u_n^2(j,x)/2},
\end{equation}
uniformly for $1\le j\le i_n$.
\end{lem}

\begin{proof}
We first consider $\tau(s)\le2\tau(1)$ and write
$p=L_n(x)/s$.
Let $E_1,\ldots,E_{n-j+1},F_1,\ldots,F_q$ be
independent exponential random variables with mean one.
The gamma representation of the beta distribution gives
\[
Y_{n-j}\overset d=
\frac{\sum_{r=1}^{n-j+1}E_r}
{\sum_{r=1}^{n-j+1}E_r+\sum_{r=1}^qF_r}.
\]
Since the denominator is positive, the event that this ratio
exceeds $p$ is equivalent to
\[
(1-p)\sum_{r=1}^{n-j+1}E_r
-p\sum_{r=1}^qF_r>0.
\]
Thus, writing
\[
W_n=(1-p)\sum_{r=1}^{n-j+1}(E_r-1)
-p\sum_{r=1}^q(F_r-1),
\]
we have
\[
\mathbb P(Y_{n-j}>p)
=\mathbb P\bigl(W_n>pq-(1-p)(n-j+1)\bigr).
\]
This is a sum of independent centered random variables,
with variance
\[
{\rm Var}(W_n)=(n-j+1)(1-p)^2+qp^2.
\]
Setting
\[
z=\frac{pq-(1-p)(n-j+1)}{\sqrt{ {\rm Var}(W_n)}},
\]
then
\[
\mathbb P(Y_{n-j}>p)=\mathbb P(\frac{W_n}{\sqrt {{\rm Var}(W_n)}}>z).
\]

We compare $z$ with $\tau(s)$.
By definition,
\[
p=\frac{n-j}{m-j}
+\frac{\sqrt{(n-j)q}}{(m-j)^{3/2}}\tau(s).
\]
It follows that
\[
pq-(1-p)(n-j+1)
=
\sqrt{\frac{(n-j)q}{m-j}}\,\tau(s)-(1-p)
\]
and
\[
{\rm Var}(W_n)=
\frac{(n-j)q}{m-j}
\Bigl(1+\frac{\tau(s)^2}{m-j}\Bigr)
+(1-p)^2.
\]
Since $0\le 1-p\le1$ and $1\le j\le i_n=o(n),$ it follows that
\[
{\rm Var}(W_n)=\frac{(n-j)q}{m-j}
\Bigl(
1+O\Bigl(\frac{\tau^2(s)}{m}+\frac{m}{nq}\Bigr)
\Bigr)
\]
and 
$$pq-(1-p)(n-j+1)
=
\sqrt{\frac{(n-j)q}{m-j}}(\tau(s)+O(\sqrt{\frac{m}{nq}})).$$ 
Thus, it follows from the definition of $z$ that 
$$z=\tau(s)(1+O(\frac{\tau^2(s)}{m}+\frac{m}{nq}))+O(\sqrt{\frac{m}{nq}}).$$
The facts 
$0<\tau(s)\le2\tau(1)\lesssim s_n^{1/30}$ and $s_n\le m^2$ imply
$$\tau^{k}(s)m^{-1}\lesssim \tau^k(1)m^{-1}\lesssim s_n^{k/30}m^{-1}\lesssim m^{-1+k/15}\ll 1$$
once $0\le k<15.$ Also, Lemma \ref{snmnq} implies $\frac{m}{nq}=o(1).$ 
Consequently,  
\[
z=\tau(s)+O\bigl(
\frac{\tau^3(s)}{m}
+\frac{\tau(s)m}{nq}
+\sqrt{\frac{m}{nq}}
\bigr).
\]
Since \(\tau'(s)=-T_1/s^2<0\), we have \(\tau(s)\ge\tau(1)\) for \(s\le1\). By Lemma \ref{T1-T2} and the monotonicity of \(u_n(j,x)\) in \(j\),
\[
\tau(1)\sim u_n(j,x)\ge u_n(1,x)
=\sqrt{\gamma_n}+\frac{x}{\sqrt{\gamma_n}}
\sim\sqrt{\gamma_n}\gg1.
\]
Thus \(\tau(s)\gg1\).
Hence $z\sim\tau(s)$ and moreover 
\[
z^2-\tau^2(s)
=O\bigl(
\frac{\tau^4(s)}{m}
+\frac{\tau^2(s)m}{nq}
+\tau(s)\sqrt{\frac{m}{nq}}
\bigr)
=o(1).
\]
Here,  the second equality is true since under the assumption $\tau(s)\le2\tau(1) $, together with Lemma \ref{snmnq} and \ref{T1-T2}, we have 
$$\frac{\tau^4(s)}{m}
+\frac{\tau^2(s)m}{nq}
+\tau(s)\sqrt{\frac{m}{nq}}\lesssim s_n^{2/15}m^{-1}+s_n^{1/30}\sqrt{\frac{m}{nq}}\ll 1.$$
We now apply a Cram\'er-type moderate deviation estimate.
Set
\[
\varepsilon=\frac{2\max\{p,1-p\}}{\sqrt {{\rm Var}(W_n)}}
\asymp\sqrt{\frac{m}{nq}}.
\]
Each summand of $W_n/\sqrt {{\rm Var}(W_n)}$ has the form
$c(E-1)$, where $E\sim\operatorname{Exp}(1)$ and
$|c|\le\varepsilon/2$.
Since $\mathbb E|E-1|^k\le k!$ for integers $k\ge2$,
\[
|\mathbb E\xi^k|
\le k!|c|^{k-2}\mathbb E\xi^2
\le\frac{k!}{2}\varepsilon^{k-2}\mathbb E\xi^2,
\qquad k\ge3.
\]
Thus the normalized summands satisfy the Bernstein
condition, and their variances sum to one.

Since \(z\sim\tau(s)\) uniformly and \(\tau(s)\le 2\tau(1)\lesssim s_n^{1/30}\), Lemma 2.6 with \(a=1/10\) gives
\[
z^3\varepsilon\lesssim \tau^3(1)\sqrt{\frac{m}{nq}}
\lesssim s_n^{1/10}\sqrt{\frac{m}{nq}}=o(1).
\]
Hence \(z=o(\varepsilon^{-1/3})\), and in particular \(z=o(\varepsilon^{-1})\). So the range
condition in \cite{FanGramaLiu2013} is satisfied.
By \cite[Theorems~2.1 and~2.2]{FanGramaLiu2013},
\[
\mathbb P(\frac{W_n}{\sqrt{{\rm Var}(W_n)}}>z)
=
(1-\Phi(z))
\exp\Bigl\{
O\bigl(z^3\varepsilon
+(1+z)\varepsilon|\log\varepsilon|\bigr)
\Bigr\}.
\]
Here $\Phi$ denotes the standard normal distribution
function.
For the second term, it follows from
\(z=o(\varepsilon^{-1/3})\) that 
\[
(1+z)\varepsilon|\log\varepsilon|
=
O(\varepsilon|\log\varepsilon|)
+o(\varepsilon^{2/3}|\log\varepsilon|)
=o(1).
\]
Therefore,
\[
\mathbb P(Y_{n-j}>p)=\mathbb P(\frac{W_n}{\sqrt {{\rm Var}(W_n)}}>z)
=(1+o(1))(1-\Phi(z)).
\]
Since $z\sim\tau(s)\gg 1$, and
$z^2-\tau^2(s)=o(1)$ uniformly, the Gaussian tail
asymptotic yields
\[
\mathbb P\bigl(Y_{n-j}>p\bigr)
=
\frac{1+o(1)}{\sqrt{2\pi}\,\tau(s)}
e^{-\tau^2(s)/2}.
\]
This proves \eqref{beta-tail-local}.

For the remaining range, observe that
\[
\tau(L_n(x))
=\sqrt{\frac{q(m-j)}{n-j}}.
\]
By Lemma \ref{T1-T2}, we have
\[
\tau(1)=T_1-T_2\sim u_n(j,x)\lesssim s_n^{1/30}
\]
uniformly for \(1\le j\le i_n\). 
%Since \(i_n=o(n)\), it follows that
%\[
%\frac{\tau(L_n(x))}{\tau(1)}
%\gtrsim
%\frac{\sqrt{qm}}{\sqrt n\,s_n^{1/30}}
%=
%\frac{m^{7/15}q^{8/15}}{n^{8/15}}
%\gtrsim
%m^{11/75}(\log m)^{16/15}
%\longrightarrow\infty
%\]
Since \(i_n=o(n)\),
\[
\tau(L_n(x))
=\sqrt{\frac{q(m-j)}{n-j}}
\asymp\sqrt{\frac{mq}{n}}
\]
uniformly for \(1\le j\le i_n\). Therefore,
\[
\frac{\tau(1)}{\tau(L_n(x))}
\lesssim
s_n^{1/30}\sqrt{\frac{n}{mq}}
\le
s_n^{1/30}\sqrt{\frac{m}{nq}}
=o(1)
\]
by Lemma \ref{snmnq}.

Hence there is a unique $s_*\in(L_n(x),1)$ such that
$\tau(s_*)=2\tau(1)$.
If $\tau(s)>2\tau(1)$, then $s<s_*$, so the monotonicity of the probability and the asymptotic \eqref{beta-tail-local} lead 
\[
\begin{aligned}
\mathbb P\bigl(Y_{n-j}>\frac{L_n(x)}s\bigr)
&\le
\mathbb P\bigl(Y_{n-j}>\frac{L_n(x)}{s_*}\bigr)\lesssim\frac{e^{-2\tau^2(1)}}{\tau(1)}.
\end{aligned}
\]
 This proves \eqref{beta-tail-far}.
 
 Finally, taking $s=1$ in \eqref{beta-tail-local}, Lemma 2.7 gives 
 $\tau(1)\sim u_n(j,x)$
 and
$ \tau^2(1)-u_n^2(j,x)=o(1)$,
 uniformly for \(1\le j\le i_n\), which yields
\eqref{Pv}.
\end{proof}

\begin{lem}\label{2m}
Let $i_n=\lfloor s_n^{8/15}\rfloor$,
$h_n=\lfloor s_n^{3/8}\rfloor$, and let $$u_n(t,x)=\frac{t-1}{\sqrt{s_n}}+\sqrt{\gamma_n}+\frac{x}{\sqrt{\gamma_n}}$$ for any $x\in\mathbb{R}$ and $t\ge 1.$ 

For fixed $x\in\mathbb R$, the following estimates hold.
\begin{enumerate}
    \item[(i)] Uniformly for $1\le j\le i_n$,
    \[
    M_{j,j}(x)
    =
    \frac{1+o(1)}
    {2\sqrt{2}\pi s_n^{1/4}u_n^{3/2}(j,x)}
    e^{-\frac12u_n^2(j,x)}.
    \]

    \item[(ii)] For $i_n\leq j\le n$,
    \[
    M_{j,j}(x)
    \lesssim n^{-3/10}e^{-\frac13n^{1/15}}.
    \]

    \item[(iii)] For $1\le j<k\le i_n$ with $k-j>h_n$,
    \[
    |M_{j,k}(x)|
    \lesssim
    \frac{e^{-\frac12u_n^2((j+k)/2,x)}}
    {(k-j)u_n((j+k)/2,x)}.
    \]
\end{enumerate}
\end{lem}

\begin{proof}
We first prove (i). By the integral representation
\eqref{wMjj2},
\[
M_{j,j}(x)
=\frac{1}{2\pi}
\int_{L_n(x)}^1
\frac{\mathbb P(Y_{n-j}>L_n(x)/s)}
{\sqrt{s(1-s)}}\,ds.
\]
Use the notation of Lemma \ref{beta-tail}, and let $s_*$
be determined by $\tau(s_*)=2\tau(1)$.
On $[s_*,1]$, the local estimate in that lemma gives
\begin{equation}\label{inte}
\int_{s_*}^1
\frac{\mathbb P(Y_{n-j}>L_n(x)/s)}
{\sqrt{s(1-s)}}\,ds
=
(1+o(1))
\int_{s_*}^1
\frac{\phi(\tau(s))}
{\tau(s)\sqrt{s(1-s)}}\,ds,
\end{equation}
where $\phi(t)=(2\pi)^{-1/2}e^{-t^2/2}$.
The relative error is uniform in both $j$ and $s$, so it
remains $o(1)$ after integration.

On $[L_n(x),s_*]$, the remaining estimate in
Lemma \ref{beta-tail} yields

\begin{equation}\label{inte1}\aligned 
\int_{L_n(x)}^{s_*}
\frac{\mathbb P(Y_{n-j}>L_n(x)/s)}
{\sqrt{s(1-s)}}\,ds
&\lesssim
\frac{e^{-2\tau^2(1)}}{\tau(1)}
\int_{L_n(x)}^{s_*}\frac{ds}{\sqrt{s(1-s)}}\\
&\lesssim \frac{e^{-2\tau^2(1)}}{\tau(1)},
\endaligned
\end{equation}

since the last integral is at most $\pi$.
Also, $t\mapsto\phi(t)/t$ is decreasing for $t>0$.
As $\tau(s)\ge2\tau(1)$ on this interval, the same argument
gives
\begin{equation}\label{inte2}\aligned 
\int_{L_n(x)}^{s_*}
\frac{\phi(\tau(s))}
{\tau(s)\sqrt{s(1-s)}}\,ds
\lesssim\frac{e^{-2\tau^2(1)}}{\tau(1)}.
\endaligned
\end{equation} 
Combining the three integrals \eqref{inte}, \eqref{inte1} and \eqref{inte2} together, we have 
\begin{equation}\label{Mjj}
M_{j,j}(x)
=\frac{1+o(1)}{2\pi}
\int_{L_n(x)}^1
\frac{\phi(\tau(s))}
{\sqrt{s(1-s)}}\,ds+O(\frac{e^{-2\tau^2(1)}}{\tau(1)}).
\end{equation}

We next evaluate the integral in \eqref{Mjj}.
By Lemma \ref{T1-T2},
\[
T_1\sim\sqrt{s_n},
\qquad
T_2\sim\sqrt{s_n},
\qquad
T_1-T_2=\tau(1)\gg1.
\]
We next verify $\frac{T_1(1-L_n(x))}{L_n(x)}\gg 1.$ In fact, \eqref{An} and \eqref{Bn} imply 
$$L_n(x)=\frac{n}{m}(1+O(\frac{\sqrt{\log s_n}}{\sqrt{s_n}}+\frac{\sqrt{\log s_n}}{n}+\frac{\sqrt{q}}{\sqrt{n m \gamma_n}})),$$
which implies 
$$L_n(x)\sim \frac{n}{m} \qquad \text{and} \qquad 1-L_n(x)\sim\frac{q}{m}.$$
Consequently,
\[
\frac{T_1(1-L_n(x))}{L_n(x)}
\sim\sqrt{s_n}\frac{q}{n}
=\sqrt{\frac{mq}{n}}
\gg1.
\]
Thus, Lemma \ref{int} is applicable with
$a=T_1$, $b=T_2$ and $v=L_n(x)$ and then 
\begin{equation}\label{Mjj0}
\int_{L_n(x)}^1
\frac{\phi(\tau(s))}
{\tau(s)\sqrt{s(1-s)}}\,ds
=
\frac{1+o(1)}{\sqrt{2T_1}\,(\tau(1))^{3/2}}
e^{-\frac12\tau^2(1)}.
\end{equation}
Comparing the expressions \eqref{Mjj} and \eqref{Mjj0}, once  
\[
\sqrt{T_1 \tau(1)}e^{-\frac32\tau^2(1)}\ll 1,
\] 
we have 
\begin{equation}\label{Mjj-semifinal}
M_{j,j}(x)
=
\frac{1+o(1)}
{2\pi\sqrt{2T_1}\,\tau^{3/2}(1)}
e^{-\frac12\tau^2(1)}.
\end{equation}

In fact, by  Lemma \ref{T1-T2} and the monotonicity of $u_n(j,x)$,
\[
\tau^2(1)
=u_n^2(j,x)+o(1)
\ge u_n^2(1,x)+o(1)
=\gamma_n+2x+o(1).
\]
Since $t^{1/2}e^{-3t^2/2}$ is decreasing for all sufficiently
large $t$, it follows that
\[
\begin{aligned}
\sqrt{T_1\tau(1)}e^{-\frac32\tau^2(1)}
&\lesssim
s_n^{1/4}\gamma_n^{1/4}e^{-\frac32\gamma_n}\le s_n^{-1/4}=o(1),
\end{aligned}
\]
where we used $\gamma_n\sim\frac12\log s_n\ge \frac{5}{12}\log s_n$.

Finally, Lemma \ref{T1-T2} gives
\[
T_1^{1/2}\sim s_n^{1/4},
\qquad
(\tau(1))^{3/2}\sim u_n^{3/2}(j,x),
\qquad
e^{-\frac12\tau^2(1)}
=(1+o(1))e^{-\frac12 u_n^2(j,x)}.
\]
Substituting these estimates into
\eqref{Mjj-semifinal}  proves (i).

For (ii), the diagonal entries are nonincreasing in $j$,
so part (i) implies
\[
M_{j,j}(x)
\le M_{i_n,i_n}(x)
\lesssim
s_n^{-1/4}u_n^{-3/2}(i_n,x)
e^{-\frac12 u_n^2(i_n,x)},
\qquad j\geq i_n.
\]
Since
\[
u_n(i_n,x)
=\frac{i_n-1}{\sqrt{s_n}}
+\sqrt{\gamma_n}+\frac{x}{\sqrt{\gamma_n}}
\sim s_n^{1/30},
\]
we have
\[
s_n^{-1/4}u_n^{-3/2}(i_n,x)
\lesssim s_n^{-3/10},
\qquad
\frac12u_n^2(i_n,x)\ge\frac13s_n^{1/15}
\]
for all sufficiently large $n$. Therefore,
\[
M_{j,j}(x)
\lesssim s_n^{-3/10}e^{-\frac13s_n^{1/15}}
\le n^{-3/10}e^{-\frac13n^{1/15}},
\]
where the last inequality follows from $s_n=mn/q\ge n$.
This proves (ii).

Finally, we prove (iii), including the case where $j+k$
is odd. The off-diagonal bound in Lemma 2.1 gives
\begin{equation}\label{Py}
|M_{j,k}(x)|
\le\frac{1}{\pi(k-j)}
\mathbb P\bigl(Y_{n-(j+k)/2}>L_n(x)\bigr),
\end{equation}
where
\[
Y_{n-(j+k)/2}
\sim\operatorname{Beta}\bigl(n-(j+k)/2+1,q\bigr).
\]
Set $\ell=\lfloor(j+k)/2\rfloor$.
The beta distribution is stochastically increasing in its
first parameter when its second parameter is fixed.
Since
\[
n-\frac{j+k}{2}+1\le n-\ell+1,
\]
we obtain
\[
\mathbb P\bigl(Y_{n-(j+k)/2}>L_n(x)\bigr)
\le\mathbb P(Y_{n-\ell}>L_n(x)).
\]
Here $1\le\ell\le i_n$, so Lemma \ref{beta-tail} yields
\[
\mathbb P(Y_{n-\ell}>L_n(x))
\lesssim
\frac{e^{-\frac12u_n^2(\ell,x)}}{u_n(\ell,x)}.
\]

It remains to replace $\ell$ by $(j+k)/2$.
By the definition of $u_n$,
\[
0\le
u_n\Bigl(\frac{j+k}{2},x\Bigr)-u_n(\ell,x)
\le\frac{1}{2\sqrt{s_n}}.
\]
Since $u_n(\ell,x)\gtrsim\sqrt{\log s_n}$, this implies
\[
\frac{u_n((j+k)/2,x)}{u_n(\ell,x)}
=
1+O\Bigl(\frac{1}{\sqrt{s_n\log s_n}}\Bigr)
=1+o(1).
\]
Moreover, both arguments are at most $i_n$, and therefore
\[
\begin{aligned}
0
\le u_n^2\bigl(\frac{j+k}{2},x\bigr)-u_n^2(\ell,x)
&=
\bigl(
u_n\bigl(\frac{j+k}{2},x\bigr)-u_n(\ell,x)
\bigr)
\bigl(
u_n\bigl(\frac{j+k}{2},x\bigr)+u_n(\ell,x)
\bigr)\\
&\lesssim s_n^{-1/2}s_n^{1/30}
=s_n^{-7/15}
=o(1).
\end{aligned}
\]
Consequently,
\[
\frac{e^{-\frac12 u_n^2(\ell,x)}}{u_n(\ell,x)}
=
(1+o(1))
\frac{e^{-\frac12 u_n^2((j+k)/2,x)}}
{u_n((j+k)/2,x)}.
\]
Combining this estimate with \eqref{Py} proves (iii).
\end{proof}

\subsection{Proof of Theorem~1 for \(m^{2/5}(\log m)^{2}\lesssim q<m\)}

We first estimate the trace. Using the diagonal estimate from Lemma~\ref{2m}(i) and Lemma \ref{sum1}, we obtain
\begin{equation}\label{sumin}
	\sum_{k=1}^{i_n}M_{k,k}(x)\sim\sum_{k=1}^{i_n}\frac{e^{-\frac{1}{2}u_n^{2}(k,x)}}{2\sqrt{2}\pi s_n^{1/4} u_n^{3/2}(k,x)}
	\sim \frac{s_n^{1/4}e^{-\frac{1}{2}u_n^{2}(1,x)}}{2\sqrt{2}\pi\, u_n^{5/2}(1,x)},
\end{equation}
where the last equivalence is due to the  fact
\[
\frac{1}{u_n^{5/2}(i_n+1,x)}e^{-\frac{u_n^{2}(i_n+1,x)}{2}}\ll \frac{e^{-\frac{1}{2}u_n^{2}(1,x)}}{u_n^{5/2}(1,x)}.
\]
Indeed,
\begin{equation}\label{un1}
	u_n(i_n+1,x)\sim s_n^{1/30} \quad \text{and}\quad u_n(1,x)\sim \sqrt{\frac{\log s_n}{2}}.
\end{equation}
Furthermore, the identity
\[
 u_n^2(1,x)=\gamma_n+2x+x^2/\gamma_n
\]
and \eqref{gamma-match} give
\begin{equation}\label{expu2}
 e^{-\frac12u_n^2(1,x)}
 =2\sqrt2\pi\gamma_n^{5/4}s_n^{-1/4}e^{-x-\frac{x^2}{2\gamma_n}}
   (1+o(\gamma_n^{-1})).
\end{equation}
Thus, the leading expression in \eqref{sumin} is
\begin{equation}\label{TrMx}
 e^{-x}\bigl(1+\frac{x}{\gamma_n}\bigr)^{-5/2}
 e^{-\frac{x^2}{2\gamma_n}}(1+o(\gamma_n^{-1}))
\end{equation} 
and then 
\begin{equation}\label{suminok}
 \sum_{k=1}^{i_n}M_{k,k}(x)\sim e^{-x}.
\end{equation}
For the tail \(j>i_n\), Lemma~\ref{2m}(ii) gives
\begin{equation}\label{TrMx1}
\sum_{j=i_n+1}^{n}M_{j,j}(x)\lesssim n^{7/10}e^{-\frac{1}{3}n^{1/15}}\ll 1.
\end{equation}
Combining this with \eqref{suminok}, we conclude that
\[
\operatorname{Tr}(M(x))=\sum_{j=1}^{i_n}M_{j,j}(x)+\sum_{j=i_n+1}^{n}M_{j,j}(x)\sim e^{-x}.
\]
It remains to estimate the Hilbert--Schmidt norm.
Since $\operatorname{Tr}(M(x))=O(1)$, we apply
Lemma~\ref{HS} with $\alpha_n=i_n$ and $\beta_n=h_n$.

First, Lemma~\ref{2m}(ii) gives
\begin{equation}\label{2pnmin1}
nM_{i_n,i_n}(x)
\lesssim n^{7/10}e^{-\frac13n^{1/15}}
=o(1).
\end{equation}
Next, Lemma~\ref{2m}(i), together with
\eqref{un1} and \eqref{expu2}, yields
\begin{equation}\label{2inhn}
\begin{aligned}
i_nh_n|M_{1,1}(x)|^2
&\lesssim
s_n^{49/120}u_n^{-3}(1,x)e^{-u_n^2(1,x)}\\
&\lesssim s_n^{-11/120}\log s_n
=o(1).
\end{aligned}
\end{equation}

For the remaining off-diagonal terms, Lemma~\ref{2m}(iii)
gives
\[
\sum_{\substack{1\le j<k\le i_n\\ k-j>h_n}}
|M_{j,k}(x)|^2
\lesssim
\sum_{j=1}^{i_n}
\sum_{k=j+h_n+1}^{i_n}
\frac{e^{-u_n^2((j+k)/2,x)}}
     {(k-j)^2u_n^2((j+k)/2,x)}.
\]
Here and below, an empty inner sum is understood to be zero.
Since $u_n(t,x)$ is increasing in $t$ and
$t^{-2}e^{-t^2}$ is decreasing for $t>0$, we have
\[
\frac{e^{-u_n^2((j+k)/2,x)}}
     {u_n^2((j+k)/2,x)}
\le
\frac{e^{-u_n^2(j,x)}}{u_n^2(j,x)}.
\]
Thus all subsequent estimates involve integer indices.
Using
\[
\sum_{k=j+h_n+1}^{i_n}(k-j)^{-2}
\le\sum_{d=h_n+1}^{\infty}d^{-2}
\le h_n^{-1},
\]
and Lemma \ref{sum1}(ii)
we obtain
\[
\sum_{\substack{1\le j<k\le i_n\\ k-j>h_n}}
|M_{j,k}(x)|^2
\lesssim
h_n^{-1}
\sum_{j=1}^{i_n}
u_n^{-2}(j,x)e^{-u_n^2(j,x)}\lesssim h_n^{-1}\sqrt{s_n}\,u_n^{-3}(1,x)e^{-u_n^2(1,x)}.
\]
Since $h_n\asymp s_n^{3/8}$, it follows that
\[
\begin{aligned}
\sum_{\substack{1\le j<k\le i_n\\ k-j>h_n}}
|M_{j,k}(x)|^2
\lesssim
s_n^{1/8}u_n^{-3}(1,x)e^{-u_n^2(1,x)}\lesssim s_n^{-3/8}\log s_n
=o(1).
\end{aligned}
\]
Combining this estimate with \eqref{2pnmin1} and
\eqref{2inhn}, Lemma~\ref{HS} gives
\[
\|M(x)\|_{\mathrm{HS}}^2=o(1),
\]
and hence $\|M(x)\|_{\mathrm{HS}}=o(1)$. The proof is completed. 

\section{The regime \(1\ll q\ll m^{2/5}(\log m)^2\)}

We first establish the estimates on the matrix entries.

\begin{lem}\label{1nkM}
	Let
	\[
	t_n=\bigl\lfloor \bigl(\tfrac{m}{q}\bigr)^{3/5}\bigr\rfloor,\qquad
	j_n=\bigl\lfloor\frac{m}{q^{3/4}\sqrt{1-\lambda_n}}\bigr\rfloor,\qquad
	\ell_n=\bigl\lfloor\frac{m}{\sqrt{q(1-\lambda_n)}}\bigr\rfloor.
	\]
	For each fixed \(x\in\mathbb R\), the following estimates hold
	uniformly over the indicated ranges of indices.

	\begin{enumerate}
		\item[(i)] For \(1\le j\le j_n\),
		$$	M_{j,j}(x)\sim\frac{(1-\lambda_n)q}{m}
		\exp\Bigl(-x-\frac{(j-1)(1-\lambda_n)q}{m}\Bigr). $$

		\item[(ii)] For \(j_n<j\le \ell_n\),
		$$M_{j,j}(x)\lesssim \frac{(1-\lambda_n)q}{m}
		\exp\Bigl(-x-\frac{(j-1)(1-\lambda_n)q}{m}\Bigr). $$
		\item[(iii)] For \(j\geq \ell_n\),
		$$	M_{j,j}(x)\lesssim \frac{q}{m}e^{-q^{1/4}}. $$

		\item[(iv)] For \(k-j>t_n\) and \(k\le \ell_n\),
		\begin{equation}\label{Mjk-offdiag}
			|M_{j,k}(x)|\lesssim
			\frac{(1-\lambda_n)^{3/2}q}{m^{9/20}(k-j)}
			\exp\Bigl(-\frac{(j+k)(1-\lambda_n)q}{2m}\Bigr).
		\end{equation}
	\end{enumerate}
\end{lem}
\begin{proof}
	Write
	\[
	\delta_n:=1-\lambda_n.
	\]
	We prove each estimate in turn.

	\emph{Proof of (i).}
	Starting from the integral representation \eqref{wMjj1} in
	Lemma~\ref{integral-rep}, we have
	\[
	M_{j,j}(x)
	=
	C_j^2\int_{L_n(x)}^{1}(1-u)^{q-1}u^{n-j}
	\arccos\Bigl(\sqrt{\frac{L_n(x)}{u}}\Bigr)\,du,
	\]
	where
	\[
	C_j=\sqrt{\frac{\Gamma(m-j+1)}
	{\pi\,\Gamma(q)\Gamma(n-j+1)}},\qquad
	L_n(x)=A_n+B_nx
	=1-\frac{q\lambda_n}{m}
	+\frac{\lambda_n x A_n}{m\delta_n}.
	\]
	Observe that \(L_n(x)\to1\) and then 
	\[
	\arccos\sqrt{\frac{L_n(x)}{u}}
	=
	\sqrt{u-L_n(x)}\,(1+O(1-L_n(x)))
	\]
	uniformly for \(u\in (L_n(x), 1)\). Consequently,
	\begin{equation}\label{Mjj1}
	M_{j,j}(x)
	\sim C_j^2
	\int_{L_n(x)}^1
	(1-u)^{q-1}u^{n-j}\sqrt{u-L_n(x)}\,du.
	\end{equation}

	Set
	\[
	w=1-L_n(x)
	=\frac{q\lambda_n}{m}
	-\frac{\lambda_n x A_n}{m\delta_n}
	=
	\frac{q\lambda_n}{m}
	\Bigl(1-\frac{xA_n}{q\delta_n}\Bigr)
	\ll1,
	\]
	where the final estimate follows from
	\(\sqrt q\,\delta_n\gg1\) and the assumed growth of \(q\). Let \(s=(1-u)/w\). Then
	\[
	u^{n-j}
	=
	\exp\{(n-j)\log(1-ws)\}
	=
	\exp\{- (n-j)ws+O(nw^2s^2)\}.
	\]
	Since
	\[
	nw^2s^2\le nw^2\lesssim \frac{nq^2}{m^2}=o(1),
	\]
	it follows that
	\[
	u^{n-j}
	=
	(1+o(1))e^{-(n-j)ws}
	\]
	uniformly for \(0\le s\le1\). Consequently,
	\begin{equation}\label{sus}
		\int_{L_n(x)}^{1}(1-u)^{q-1}u^{n-j}
		\sqrt{u-L_n(x)}\,du
		\sim
		w^{q+1/2}
		\int_{0}^{1}s^{q-1}e^{-(n-j)ws}\sqrt{1-s}\,ds.
	\end{equation}
	Next, with the change of variables \(t=q(1-s)\), we obtain
	\begin{equation}\label{sus2}
		\int_{0}^{1}s^{q-1}e^{-(n-j)ws}\sqrt{1-s}\,ds
		=
		q^{-3/2}e^{-(n-j)w}
		\int_{0}^{q}
		\Bigl(1-\frac{t}{q}\Bigr)^{q-1}
		e^{\frac{(n-j)w}{q}t}\sqrt{t}\,dt.
	\end{equation}

	For \(1\le j\le j_n\), set
	\[
	a_j=1-\frac{(n-j)w}{q}.
	\]
	By the definition of \(w\) and the identity \(n=m-q\),
	\[
	a_j
	=
	\delta_n
	+\frac{\lambda_n(q+j)}m
	+\frac{\lambda_n xA_n(n-j)}{mq\delta_n}.
	\]
	Consequently,
	\[
	\frac{a_j}{\delta_n}
	=
	1+O\Bigl(
	\frac{q+j_n}{m\delta_n}
	+\frac1{q\delta_n^2}
	\Bigr).
	\]
	The error tends to zero. Indeed,
	\[
	\frac{j_n}{m\delta_n}
	\le
	\frac1{q^{3/4}\delta_n^{3/2}}
	=o(1),
	\]
	while
	\[
	\frac{q}{m\delta_n}
	\lesssim\frac{q^{3/2}}m=o(1),
	\qquad
	\frac1{q\delta_n^2}=o(1).
	\]
	Here we used \(\sqrt q\,\delta_n\to\infty\) and
	\(q\ll m^{2/5}(\log m)^2\).
	Thus, uniformly for \(1\le j\le j_n\),
	\begin{equation}\label{pnjw}
	a_j=\delta_n(1+o(1)).
	\end{equation}

	Choose the cutoff
	\[
	\xi=q^{1/4}a_j^{-1/2}.
	\]
	Then
	\[
	\frac{\xi^2}{q}
	=\frac1{\sqrt q\,a_j}=o(1),
	\qquad
	a_j\xi=(\sqrt q\,a_j)^{1/2}\longrightarrow\infty.
	\]
	We split the integral in \eqref{sus2} at \(\xi\).
	For \(0\le t\le\xi\), Taylor expansion gives
	\[
	(q-1)\log(1-t/q)
	=-t+O(t/q+t^2/q)
	=-t+o(1)
	\]
	uniformly. Hence, with the substitution \(u=a_jt\),
	\begin{equation}\label{sus3}
	\begin{aligned}
	\int_0^\xi
	\Bigl(1-\frac tq\Bigr)^{q-1}
	e^{(n-j)wt/q}t^{1/2}\,dt
	&\sim\int_0^\xi e^{-a_jt}t^{1/2}\,dt\\
	&=a_j^{-3/2}
	\int_0^{a_j\xi}e^{-u}u^{1/2}\,du\\
	&\sim\frac{\sqrt\pi}{2}\delta_n^{-3/2},
	\end{aligned}
	\end{equation}
	where the last step follows from \(a_j\xi\to\infty\)
	and \eqref{pnjw}.

	To estimate the tail, note that, for \(0\le t<q\),
	\begin{equation}\label{tq}
	\bigl(1-\frac tq\bigr)^{q-1}
	=
	\exp\bigl\{(q-1)\log\bigl(1-\frac tq\bigr)\bigr\}
	\le e^{-t+t/q}
	\le e^{-t+1}.
	\end{equation}
	Hence, with the substitution \(u=a_jt\),
	\[
	\begin{aligned}
	\int_\xi^q
	\bigl(1-\frac tq\bigr)^{q-1}
	e^{(n-j)wt/q}t^{1/2}\,dt
	&\lesssim \int_\xi^q e^{-a_jt}t^{1/2}\,dt\\
	&\le a_j^{-3/2}
	\int_{a_j\xi}^\infty e^{-u}u^{1/2}\,du\\
	&\ll \delta_n^{-3/2},
	\end{aligned}
	\]
	since \(a_j\xi\to\infty\) and \(a_j\sim\delta_n\).
	Combining this bound with \eqref{sus3}, we obtain
	\[
	M_{j,j}(x)
	\sim
	\frac{\sqrt{\pi}}2\,C_j^2w^{q+1/2}
	q^{-3/2}e^{-(n-j)w}\delta_n^{-3/2}.
	\]

	Stirling's formula and Taylor expansion give
	\begin{align}
	\sqrt{\pi}\,C_j^2
	&=
	\frac{(1+o(1))e^{q}(m-j)^{q}}
	{e\sqrt{2}\pi(q-1)^{q-1/2}},
	\label{em1}\\
	w^{q+1/2}
	&\sim
	\bigl(\frac{q}{m}\bigr)^{q+1/2}
	e^{(q+1/2)\log\lambda_n-\frac{xA_n}{\delta_n}},
	\label{em2}\\
	(n-j)w
	&=
	q\lambda_n-\frac{\lambda_n xA_n}{\delta_n}
	-\frac{q\lambda_n j}{m}+o(1),
	\label{em3}\\
	\bigl(\frac{m-j}{m}\bigr)^{q}
	&=
	\exp\bigl\{-\frac{jq}{m}
	+O(q^{-1/2}\delta_n^{-1})\bigr\}
	\sim \exp\bigl\{-\frac{jq}{m}\bigr\},
	\label{em4}
	\end{align}
	where \eqref{em4} holds for \(j\le j_n\).
	By \eqref{limequation}, there exists a remainder \(E_n=o(1)\)
	such that
	\begin{equation}\label{em5}
	e^{-q\lambda_n+q+(q+1/2)\log\lambda_n}
	=
	2\sqrt{2}\pi\delta_n^{5/2}
	m^{-1/2}q^{3/2}e^{-E_n}.
	\end{equation}
	Combining these estimates, we obtain
	\[
	M_{j,j}(x)
	=
	\frac{(1+o(1))\delta_n q}{m}
	\exp\bigl(
	-x-\frac{(j-1)\delta_n q}{m}
	\bigr).
	\]
	This proves (i).

	\emph{Proof of (ii).}
	All estimates below are uniform for \(j_n<j\le \ell_n\),
	with \(x\) fixed.
	It follows from \eqref{Mjj1}--\eqref{tq} in the proof of (i),
	using
	\[
	\bigl(1-\frac tq\bigr)^{q-1}\lesssim e^{-t},
	\]
	that
	\[
	M_{j,j}(x)
	\lesssim
	C_j^2w^{q+1/2}q^{-3/2}
	e^{-(n-j)w}a_j^{-3/2}.
	\]
	Since
	\[
	a_j
	=
	\delta_n+\frac{\lambda_n(q+j)}m
	+\frac{\lambda_n xA_n}{q\delta_n}
	(1-\frac{q+j}{m})
	\ge
	\delta_n-\frac{|x|}{q\delta_n}
	=
	\delta_n
	(1-\frac{|x|}{q\delta_n^2}),
	\]
	and \(q\delta_n^2\to\infty\), we have
	\(a_j\ge\delta_n/2\) for all sufficiently large \(m\).
	Hence
	\[
	M_{j,j}(x)
	\lesssim
	C_j^2w^{q+1/2}q^{-3/2}
	e^{-(n-j)w}\delta_n^{-3/2}.
	\]

	To control the terms involving \(j\), we use the exact identity
	\[
	-(n-j)w
	=
	-q\lambda_n
	+\frac{q\lambda_n(q+j)}m
	+\frac{\lambda_n xA_n}{\delta_n}
	-\frac{\lambda_n xA_n(q+j)}{m\delta_n}
	\]
	and the bound
	\[
	(\frac{m-j}{m})^q
	\le
	\exp\big\{
	-\frac{jq}{m}
	-\frac{qj^2}{2m^2}
	\big\}.
	\]
	Since \(0<\lambda_nA_n<1\), completing the square gives
	\begin{align*}
	-\frac{qj^2}{2m^2}
	-\frac{\lambda_n xA_nj}{m\delta_n}
	&\le
	-\frac{qj^2}{2m^2}
	+\frac{|x|j}{m\delta_n}\le
	\frac{x^2}{2q\delta_n^2}
	=o(1).
	\end{align*}
	Moreover, the assumption \(1\ll q\ll m^{2/5}(\log m)^2\) implies 
	\[
	\frac{q^2}{m}=o(1),
	\qquad
	\frac{q}{m\delta_n}=o(1)
	\qquad \text{and} \qquad
	A_n\to1.
	\]
	Combining these estimates with \eqref{em1}, \eqref{em2},
	and \eqref{em5}, we obtain
	\[
	M_{j,j}(x)
	\lesssim
	\frac{q\delta_n}{m}
	\exp\{
	-x-\frac{jq\delta_n}{m}+o(1)
	\}
	\lesssim
	\frac{q\delta_n}{m}
	\exp\{
	-x-\frac{(j-1)q\delta_n}{m}
	\},
	\]
	which proves (ii).

	\emph{Proof of (iii).}
	For \(j\geq \ell_n\), the monotonicity of \(M_{j,j}(x)\) in \(j\)
	gives
	\[
	M_{j,j}(x)
	\lesssim
	\frac{q}{m}
	\exp\Bigl(-\frac{\ell_n\delta_n q}{m}\Bigr).
	\]
	Since
	\[
	\frac{\ell_n\delta_n q}{m}
	\sim
	\sqrt{\delta_n q}\gg q^{1/4},
	\]
	we obtain
	\[
	M_{j,j}(x)
	\lesssim
	\frac{q}{m}e^{-q^{1/4}}.
	\]

	\emph{Proof of (iv).}
	Set for simplicity 
	\[
	r=\frac{j+k}{2} \quad \text{and}\quad \ell=\lfloor r\rfloor.
	\]
	By Lemma~\ref{integral-rep} and the stochastic monotonicity
	argument used in the proof of Lemma~\ref{2m}(iii),
	\[
	\begin{aligned}
	|M_{j,k}(x)|
	&\le
	\frac{1}{\pi(k-j)}
	\mathbb P\bigl(
	\operatorname{Beta}(n-r+1,q)>L_n(x)
	\bigr)\\
	&\le
	\frac{1}{\pi(k-j)}
	\mathbb P\bigl(
	\operatorname{Beta}(n-\ell+1,q)>L_n(x)
	\bigr).
	\end{aligned}
	\]
	Therefore, it suffices to estimate the last probability, whose first parameter is an integer.

	With \(w=1-L_n(x)\) as above, the changes of variables
	used in \eqref{sus} and \eqref{sus2}, now without the
	angular factor, together with \eqref{tq} and
	\(a_\ell\ge\delta_n/2\), give
	\[
	\begin{aligned}
	\int_{L_n(x)}^1
	(1-u)^{q-1}u^{n-\ell}\,du
	&\le
	w^q q^{-1}e^{-(n-\ell)w}
	\int_0^q
	\Bigl(1-\frac tq\Bigr)^{q-1}
	e^{(n-\ell)wt/q}\,dt\\
	&\lesssim
	\frac{w^q e^{-(n-\ell)w}}
	{q\delta_n}.
	\end{aligned}
	\]
	Combining this bound with the normalization estimates
	used in parts (i) and (ii), we obtain
	\[
	\mathbb P\bigl(
	\operatorname{Beta}(n-\ell+1,q)>L_n(x)
	\bigr)
	\lesssim
	\frac{\delta_n^{3/2}q}
	{\sqrt{m\lambda_n}}
	\exp\bigl(
	-x-\frac{(\ell-1)\delta_n q}{m}
	\bigr).
	\]
	Since \(0\le r-\ell<1\) and
	\[
	\frac{q\delta_n}{m}\to0,
	\]
	replacing \(\ell-1\) by \(r\) in the exponential changes
	the bound by a factor \(1+o(1)\). Thus
	\[
	|M_{j,k}(x)|
	\lesssim
	\frac{\delta_n^{3/2}q}
	{\sqrt{m\lambda_n}(k-j)}
	\exp\bigl(
	-x-\frac{(j+k)\delta_n q}{2m}
	\bigr).
	\]
	Finally, \(m\lambda_n\gg m^{9/10}\) by
	Lemma~\ref{nlambda-lower}. Absorbing the fixed factor
	\(e^{-x}\) into the implicit constant proves
	\eqref{Mjk-offdiag} for both parities of \(j+k\).
\end{proof}

\subsection{Proof of Theorem~1 for \(1\ll q\ll m^{2/5}(\log m)^2\)}
As in Section~3, we are going to prove
\[
\operatorname{Tr}(M(x))\sim e^{-x}\qquad\text{and}\qquad \|M(x)\|_{\mathrm{HS}}=o(1).
\]
Lemma~\ref{1nkM} gives, for \(j\ge \ell_n\),
\[
M_{j,j}(x)\lesssim \frac{q}{m}e^{-q^{1/4}},
\]
whence
\begin{equation}\label{trace1}
	\sum_{j=\ell_n+1}^{n}M_{j,j}(x)\lesssim q e^{-q^{1/4}}\ll 1.
\end{equation}

For \(1\le j\le j_n\), part (i) of Lemma~\ref{1nkM} yields
\begin{equation}\label{trace2}
	\sum_{j=1}^{j_n}M_{j,j}(x)
	\sim (1-\lambda_n)q m^{-1}e^{-x}
	\sum_{j=1}^{j_n}\exp\bigl(-\frac{(j-1)(1-\lambda_n)q}{m}\bigr).
\end{equation}
Since
\begin{equation}\label{jne}
	\frac{(j_n-1)(1-\lambda_n)q}{m}
	\sim q^{1/4}\sqrt{1-\lambda_n}\gg 1,
\end{equation}
the geometric series formula gives
\[
\sum_{j=1}^{j_n}\exp\bigl(-\frac{(j-1)(1-\lambda_n)q}{m}\bigr)
\sim \frac{m}{(1-\lambda_n)q}.
\]
Substituting this into \eqref{trace2}, we obtain
\begin{equation}\label{jn}
	\sum_{j=1}^{j_n}M_{j,j}(x)\sim e^{-x}.
\end{equation}

For \(j_n<j\le \ell_n\), part (ii) of Lemma~\ref{1nkM} together with \eqref{jne} implies
\[
\begin{aligned}
	\sum_{j=j_n+1}^{\ell_n}M_{j,j}(x)
	&\lesssim (1-\lambda_n)q m^{-1}e^{-x}
	\sum_{j=j_n+1}^{\ell_n}\exp\bigl(-\frac{(j-1)(1-\lambda_n)q}{m}\bigr)\\
	&\lesssim e^{-x}\exp\bigl(-\frac{j_n(1-\lambda_n)q}{m}\bigr)\ll 1.
\end{aligned}
\]
Combining this with \eqref{trace1} and \eqref{jn}, we conclude that
\begin{equation}\label{trace}
	\operatorname{Tr}(M(x))=\sum_{j=1}^{n}M_{j,j}(x)\sim e^{-x}.
\end{equation}
We now estimate the Hilbert--Schmidt norm.
Since $\operatorname{Tr}(M(x))=O(1)$ by \eqref{trace},
we apply Lemma~\ref{HS} with $\alpha_n=\ell_n$
and $\beta_n=t_n$.

We first consider the off-diagonal terms with
$k-j>t_n$. By Lemma~\ref{1nkM}(iv),
\[
\begin{aligned}
\sum_{\substack{1\le j<k\le \ell_n\\ k-j>t_n}}
|M_{j,k}(x)|^2
&\lesssim
\frac{(1-\lambda_n)^3q^2}{m^{9/10}}
\sum_{j=1}^{\ell_n}
\sum_{k=j+t_n+1}^{\ell_n}
\frac{1}{(k-j)^2}
\exp\bigl(
-\frac{(j+k)(1-\lambda_n)q}{m}
\bigr)\\
&\le
\frac{(1-\lambda_n)^3q^2}{m^{9/10}}
\sum_{j=1}^{\ell_n}
\exp\bigl(
-\frac{2j(1-\lambda_n)q}{m}
\bigr)
\sum_{k=j+t_n+1}^{\ell_n}\frac1{(k-j)^2}.
\end{aligned}
\]
An empty inner sum is understood to be zero.
Since
\[
\sum_{k=j+t_n+1}^{\ell_n}\frac1{(k-j)^2}
\le\sum_{d=t_n+1}^{\infty}d^{-2}
\le t_n^{-1},
\]
and $q(1-\lambda_n)/m\to0$, the geometric series gives
\[
\begin{aligned}
\sum_{j=1}^{\ell_n}
\exp\bigl(
-\frac{2j(1-\lambda_n)q}{m}
\bigr)
&\le
\frac{1}
{\exp\{2(1-\lambda_n)q/m\}-1}\\
&\lesssim
\frac{m}{q(1-\lambda_n)}.
\end{aligned}
\]
Consequently,
\begin{equation}\label{HS11}
\begin{aligned}
\sum_{\substack{1\le j<k\le \ell_n\\ k-j>t_n}}
|M_{j,k}(x)|^2
&\lesssim
\frac{(1-\lambda_n)^2q\,m^{1/10}}{t_n}\\
&\lesssim
\frac{q^{3/5}\log m}{m^{1/2}}
=o(1),
\end{aligned}
\end{equation}
where we used $q(1-\lambda_n)^2\lesssim\log m$,
$t_n\asymp(m/q)^{3/5}$, and
$q\ll m^{2/5}(\log m)^2$.

For the diagonal tail, Lemma~\ref{1nkM}(iii)
yields
\begin{equation}\label{HS12}
nM_{\ell_n, \ell_n}(x)
\lesssim
\frac{nq}{m}e^{-q^{1/4}}
\lesssim q e^{-q^{1/4}}
=o(1).
\end{equation}
Finally, Lemma~\ref{1nkM}(i) and the definitions of
$\ell_n$ and $t_n$ imply
\[
\begin{aligned}
\ell_n t_n|M_{1,1}(x)|^2
&\lesssim
\frac{m}{\sqrt{q(1-\lambda_n)}}
\bigl(\frac mq\bigr)^{3/5}
\bigl(\frac{(1-\lambda_n)q}{m}\bigr)^2\\
&=
\frac{q^{9/10}(1-\lambda_n)^{3/2}}{m^{2/5}}\\
&\le
\frac{q^{9/10}}{m^{2/5}}
=o(1).
\end{aligned}
\]
Combining these estimates with \eqref{HS11},
\eqref{HS12}, and Lemma~\ref{HS}, we conclude that
\[
\|M(x)\|_{\mathrm{HS}}=o(1).
\]
This completes the proof of Theorem~1 in this regime.

\section{The case of fixed \(q\)}

The task of this section is to verify 
$$
\operatorname{Tr}(M(x))\sim
\bigl(1-\frac{x}{q}\bigr)^{q+\frac12}
$$
and $\|M(x)\|_{\rm HS}=o(1).$
Throughout this section, $q$ is a fixed positive integer, and
$\lambda_n$ is defined by \eqref{equation}.
We first determine its asymptotic form:
\begin{equation}\label{fixed-U-expansion}
\lambda_n
\sim
\bigl(
\frac{2(q+1)\sqrt{\pi}\,\Gamma(q+\frac32)}
     {m^{1/2}q^{q+\frac12}}
\bigr)^{\frac{2}{2q+1}}.
\end{equation}
In particular, $\lambda_n\asymp m^{-1/(2q+1)}$.

To see this, note that the right-hand side of
\eqref{equation} tends to infinity, while its left-hand
side is strictly decreasing in $\lambda$. Since
$mq/n\to q$, it follows that $\lambda_n\to0$.
The defining equation then gives
$-\log\lambda_n=O(\log m)$.
Using $mq/n=q+O(m^{-1})$ and expanding the terms
that vanish as $\lambda_n\to0$, we obtain
\[
-\bigl(q+\frac12\bigr)\log\lambda_n
=
\frac12\log m
+\log\frac{q^{q+\frac12}}
           {2(q+1)\sqrt{\pi}\,\Gamma(q+\frac32)}
+O\bigl(\lambda_n+\frac{\log m}{m}\bigr).
\]
The remainder tends to zero, so exponentiation yields
\eqref{fixed-U-expansion}.

We use \eqref{fixed-U-expansion}, keeping the gamma factor exact,
to estimate the entries of \(M(x)\).
\begin{lem}\label{3m}
For each fixed \(x<q\), the following estimates hold uniformly over
the indicated ranges of indices.
	\begin{enumerate}
		\item[(i)] For \(1\le j\le n-\lfloor n^{3/5}\rfloor\),
		\[
		M_{j,j}(x)\sim
		\frac{q+1}{m^{q+1}}(n-j+1)^{q}
		\bigl(1-\frac{x}{q}\bigr)^{q+\frac12}.
		\]

		\item[(ii)] For \(1\le j,k\le n\),
		$$|M_{j,k}(x)|\lesssim \frac{1}{m}. $$
	\end{enumerate}
\end{lem}

\begin{proof}
	Since \(q\) is fixed, we have
	\[
	L_n(x)=A_n+B_nx=1-\frac{q\lambda_n}{m}\bigl(1-\frac{x}{q}\bigr)(1+O(\lambda_n))\sim 1.
	\]
	Proceeding as in the proof of Lemma~\ref{1nkM} and applying \eqref{sus}, we obtain
	\[
	M_{j,j}(x)=(1+o(1))C_j^2(1-L_n(x))^{q+\frac12}
	\int_{0}^{1}s^{q-1}e^{-(n-j)(1-L_n(x))s}\sqrt{1-s}\,ds.
	\]
	Observe that
	\begin{equation}\label{1-Ln}
		1-L_n(x)=\frac{q\lambda_n}{m}\Bigl(1-\frac{x}{q}\Bigr)(1+O(\lambda_n)).
	\end{equation}
	Since \(\lambda_n\ll 1\), we have
	$$(n-j)(1-L_n(x))s\lesssim \lambda_n\ll 1. $$
	Consequently,
	\[
	\int_{0}^{1}s^{q-1}e^{-(n-j)(1-L_n(x))s}\sqrt{1-s}\,ds
	\sim
	\int_{0}^{1}s^{q-1}\sqrt{1-s}\,ds
	=
	\frac{\sqrt{\pi}}{2}\frac{\Gamma(q)}{\Gamma(q+\frac32)}.
	\]
	Recalling
	\[
	C_j=\sqrt{\frac{\Gamma(m-j+1)}{\pi\Gamma(q)\Gamma(n-j+1)}},
	\]
	we obtain
	\[
	M_{j,j}(x)\sim
	\frac{\Gamma(m-j+1)}{2\sqrt{\pi}\,\Gamma(q+\frac32)\Gamma(n-j+1)}
	\Bigl(\frac{q\lambda_n}{m}\Bigl(1-\frac{x}{q}\Bigr)\Bigr)^{q+\frac12}.
	\]
	Substituting the asymptotic expression for \(\lambda_n\) gives
	\[
	M_{j,j}(x)\sim
	\frac{(m-j)!\,(q+1)}{(n-j)!\,m^{q+1}}
	\Bigl(1-\frac{x}{q}\Bigr)^{q+\frac12}.
	\]
	For \(1\le j\le n-\lfloor n^{3/5}\rfloor\), we have \(n-j\gg 1\), and hence
	\[
	\frac{(m-j)!}{(n-j)!}
	=
	(n-j+1)^{q}
	\prod_{\ell=1}^{q}\Bigl(1+\frac{\ell-1}{n-j+1}\Bigr)
	\sim (n-j+1)^{q}.
	\]
	Therefore,
	\[
	M_{j,j}(x)\sim
	\frac{q+1}{m^{q+1}}(n-j+1)^{q}
	\Bigl(1-\frac{x}{q}\Bigr)^{q+\frac12}.
	\]
	This proves (i).

	For arbitrary entries, \eqref{CSIneq} and the monotonicity of \(M_{j,j}(x)\) in \(j\) give
	\[
	|M_{j,k}(x)|\le \sqrt{M_{j,j}(x)M_{k,k}(x)}\le M_{1,1}(x)\lesssim \frac{1}{m}.
	\]
	This proves (ii) and completes the proof.
\end{proof}

\begin{lem}\label{HS3}
For fixed $q$ and $x<q$, we have
\[
\|M(x)\|_{\mathrm{HS}}=o(1).
\]
\end{lem}

\begin{proof}
We split the sum defining the squared Hilbert--Schmidt
norm according to the size of $|j-k|$:
\[
\|M(x)\|_{\mathrm{HS}}^2
=
\sum_{|j-k|<\lfloor n^{4/5}\rfloor}|M_{j,k}(x)|^2
+
\sum_{|j-k|\ge\lfloor n^{4/5}\rfloor}|M_{j,k}(x)|^2.
\]
By Lemma~\ref{3m}(ii), $|M_{j,k}(x)|\lesssim m^{-1}$.
Since the first sum contains $O(n^{9/5})$ terms and
$n\sim m$, we obtain
\[
\sum_{|j-k|<\lfloor n^{4/5}\rfloor}|M_{j,k}(x)|^2
\lesssim n^{9/5}m^{-2}
=o(1).
\]

For the second sum, Lemma~\ref{integral-rep} gives,
for $j\ne k$,
\[
|M_{j,k}(x)|
\le
\frac{C_{(j+k)/2}^2}{|k-j|}
\int_{L_n(x)}^1
(1-u)^{q-1}u^{n-(j+k)/2}\,du.
\]
Because $q$ is a fixed positive integer,
\[
\begin{aligned}
C_{(j+k)/2}^2
&=
\frac{\Gamma(m-(j+k)/2+1)}
     {\pi\Gamma(q)\Gamma(n-(j+k)/2+1)}\\
&=
\frac1{\pi\Gamma(q)}
\prod_{\ell=1}^{q}
\Bigl(n-\frac{j+k}{2}+\ell\Bigr)
\lesssim m^q.
\end{aligned}
\]
This identity holds for both integer and half-integer
values of $(j+k)/2$.
Moreover, $n-(j+k)/2\ge0$, so
\[
\int_{L_n(x)}^1
(1-u)^{q-1}u^{n-(j+k)/2}\,du
\le
\frac{(1-L_n(x))^q}{q}.
\]
Using \eqref{1-Ln}, we therefore obtain
\[
|M_{j,k}(x)|
\lesssim
\frac{m^q(1-L_n(x))^q}{|k-j|}
\lesssim
\frac{\lambda_n^q}{|k-j|}.
\]
Consequently, using the asymptotic $\lambda_n\asymp m^{-1/(2q+1)}$, we have 
\[
\begin{aligned}
\sum_{|j-k|\ge\lfloor n^{4/5}\rfloor}|M_{j,k}(x)|^2
&\lesssim
n\lambda_n^{2q}
\sum_{d=\lfloor n^{4/5}\rfloor}^{\infty}d^{-2}\asymp
m^{\,1/5-2q/(2q+1)}
=o(1).
\end{aligned}
\]

Combining the two estimates proves the lemma.
\end{proof}

\subsection{Proof of Theorem~1 for fixed \(q\)}

Recall that
\[
L_n(x)=A_n+B_nx
=1-\frac{q\lambda_n}{m}
\Bigl(1-\frac{A_nx}{q(1-\lambda_n)}\Bigr).
\]
Since $A_n-(1-\lambda_n)=n\lambda_n/m>0$, we have
$L_n(x)>1$ for $x\ge q$. Together with $R_n^2\le1$,
this implies
\[
\mathbb P(\widehat X_n\le x)=1,
\qquad x\ge q.
\]
% Since\[
% A(x)=\{z\in\mathbb{D}: \Re z>\sqrt{L_n(x)}\},
% \]
% we have \(A(x)=\varnothing\). By \eqref{MJkdefine},
% \[
% M_{j,k}(x)=0\qquad\text{for all }1\le j,k\le n.
% \]

It remains to consider \(x<q\). By Lemma~\ref{HS3},
\[
\|M(x)\|_{\mathrm{HS}}\ll 1.
\]
Thus, it suffices to estimate the trace. Write
\[
\operatorname{Tr}(M(x))=
\sum_{j=1}^{n-\lfloor n^{3/5}\rfloor}M_{j,j}(x)
+
\sum_{j=n-\lfloor n^{3/5}\rfloor+1}^{n}M_{j,j}(x).
\]
By Lemma~\ref{3m}(ii), the tail satisfies
\begin{equation}\label{mjjll1}
	\sum_{j=n-\lfloor n^{3/5}\rfloor+1}^{n}M_{j,j}(x)
	\lesssim n^{3/5}m^{-1}
	\lesssim m^{-2/5}\ll 1.
\end{equation}
For the main term, Lemma~\ref{3m} gives
\[
\sum_{j=1}^{n-\lfloor n^{3/5}\rfloor}M_{j,j}(x)
\sim
\frac{q+1}{m}
\Bigl(1-\frac{x}{q}\Bigr)^{q+\frac12}
\sum_{j=1}^{n-\lfloor n^{3/5}\rfloor}
\Bigl(\frac{n-j+1}{m}\Bigr)^{q}.
\]
Changing the summation index to \(i=n-j+1\), we have
\[
\sum_{j=1}^{n-\lfloor n^{3/5}\rfloor}
\Bigl(\frac{n-j+1}{m}\Bigr)^{q}
=
\sum_{i=\lfloor n^{3/5}\rfloor+1}^{n}
\Bigl(\frac{i}{m}\Bigr)^{q}.
\]
Since \(n/m\to 1\) and \(\lfloor n^{3/5}\rfloor/m\to 0\), a Riemann sum approximation yields
\[
\frac1m\sum_{i=\lfloor n^{3/5}\rfloor+1}^{n}
\Bigl(\frac{i}{m}\Bigr)^{q}
\sim
\int_0^1 t^{q}\,dt
=
\frac1{q+1}.
\]
Thus
\[
\sum_{i=\lfloor n^{3/5}\rfloor+1}^{n}
\Bigl(\frac{i}{m}\Bigr)^{q}
\sim
\frac{m}{q+1}.
\]
Substituting this into the preceding expression gives
\[
\sum_{j=1}^{n-\lfloor n^{3/5}\rfloor}M_{j,j}(x)
\sim
\Bigl(1-\frac{x}{q}\Bigr)^{q+\frac12}.
\]
Together with \eqref{mjjll1}, this yields
\[
\operatorname{Tr}(M(x))\sim
\bigl(1-\frac{x}{q}\bigr)^{q+\frac12}.
\]
Combining this trace limit with Lemma~\ref{HS3},
\eqref{trhs}, and \eqref{prob}, we obtain, for $x<q$,
\[
\lim_{n\to\infty}\mathbb P(\widehat X_n\le x)
=
\exp\bigl\{-\bigl(1-\frac{x}{q}\bigr)^{q+\frac12}\bigr\}.
\]
Together with the case $x\ge q$ considered above, this
establishes the limiting distribution of $\widehat X_n$.
The asymptotic equivalence of the two normalizations gives
the asserted limit for $X_n$, completing the proof
of Theorem~1.
\subsection{Proof of Remark 1.2}
	Throughout, \(x\) is fixed, and the implicit constants may depend on
	\(x\).
	
	We first consider \(q\to\infty\). In the regime of Section~3,
	we obtain from \eqref{TrMx} and \eqref{TrMx1} that 
	\[
	\operatorname{Tr}(M(x))
	=
	e^{-x}
	\Bigl(1+\frac{x}{\gamma_n}\Bigr)^{-5/2}
	e^{-\frac{x^2}{2\gamma_n}}
	\bigl(1+o(\gamma_n^{-1})\bigr).
	\]
	Hence
	\[
	\operatorname{Tr}(M(x))
	=
	e^{-x}+O(\gamma_n^{-1}).
	\]
	Since
	\[
	\gamma_n\asymp\log s_n,
	\qquad
	s_n=\frac{mn}{q}
	=n\Bigl(1+\frac nq\Bigr),
	\]
	we have \(\log s_n\asymp\log n\), and therefore
	\begin{equation}\label{rate-large-trace}
		\operatorname{Tr}(M(x))
		=
		e^{-x}+O((\log n)^{-1}).
	\end{equation}
	
	We next consider the regime of Section~4. Write
	\[
	\delta_n=1-\lambda_n,
	\qquad
	r_n=\frac{1}{q\delta_n^2}.
	\]
	By \eqref{sus2}-\eqref{sus3}, with $a_j\sim\delta_n$ from \eqref{pnjw}, the main contribution comes from $t\asymp\delta_n^{-1}$. On this scale,
	\[
	(q-1)\log\Bigl(1-\frac tq\Bigr)
	=
	-t+\frac tq-\frac{t^2}{2q}
	+O\Bigl(\frac{t^2+t^3}{q^2}\Bigr).
	\]
	Hence
	\[
	\frac{t^2}{q}\asymp r_n,
	\qquad
	\frac{t}{q}=O(r_n),
	\qquad
	\frac{t^2+t^3}{q^2}=o(r_n),
	\]
	where \(r_n=(q\delta_n^2)^{-1}\) and we used 
	\(\sqrt q\,\delta_n\to\infty\) from \eqref{eq:result}.
	Thus the quadratic term determines the first-order correction,
	while the remaining terms are of order \(r_n\) or smaller.
	The errors arising from the prefactor expansions \eqref{em1}-\eqref{em5} and the subsequent
	summation \eqref{trace2}-\eqref{jn} are of the same or smaller order.
	Consequently,
	\begin{equation}\label{rate-middle-trace}
		\operatorname{Tr}M(x)
		=
		e^{-x}+O(r_n).
	\end{equation}
	
	We now determine the order of \(r_n\). If
	\(
	1\ll q\lesssim\log n,
	\)
	then
	\(
	\delta_n\asymp1.
	\)
	Indeed, otherwise there would exist a subsequence along which
	\(\delta_n\to0\). Along this subsequence,
	\[
	q(\lambda_n-1-\log\lambda_n)
	=
	O(q\delta_n^2)
	=
	o(\log n),
	\]
	while \(-\log\lambda_n=O(\delta_n)\), and the term
	\(\log\delta_n\) in the left-hand side of \eqref{limequation}
	is nonpositive. Thus the left-hand side of \eqref{limequation} is
	bounded above by \(o(\log n)\), whereas
	\[
	\frac12\log\frac{m}{8\pi^2q^3}
	=
	\frac12\log n+O(\log\log n),
	\]
	a contradiction. Hence
	\begin{equation}\label{rate-small-q}
		q\delta_n^2\asymp q,
		\qquad
		r_n\asymp q^{-1}.
	\end{equation}
	
	Now suppose
	\[
	\log n\ll q\ll m^{2/5}(\log m)^2.
	\]
	 In this case \eqref{limequation} implies
	\(\delta_n\to0\), since otherwise
	\(q(\lambda_n-1-\log\lambda_n)\) would dominate its
	\(O(\log n)\) right-hand side.
	
	Set
	\[
	y_n=q\delta_n^2.
	\]
	Expanding \eqref{limequation} at \(\lambda_n=1\) gives
	\[
	y_n+\frac52\log y_n
	=
	\log\frac{m}{\sqrt q}
	+O(1+q\delta_n^3).
	\]
	By Lemma \ref{nlambda-lower}, \(y_n\to\infty\), and
	\[
	q\delta_n^3=y_n\delta_n=o(y_n).
	\]
	Moreover,
	\[
	\log\frac{m}{\sqrt q}\asymp\log n
	\]
	under the assumptions of Section~4. It follows that
	$$q\delta_n^2\asymp\log n,
	\qquad
	r_n\asymp(\log n)^{-1}. $$
	Together with \eqref{rate-middle-trace} and
	\eqref{rate-large-trace}, this proves the first two trace error
	bounds.
	
	Finally, let \(q\) be fixed. By \eqref{fixed-U-expansion},
	\[
	\lambda_n\asymp m^{-1/(2q+1)}.
	\]
	Retaining one further term in the expansions of Section~5 gives,
	uniformly for
	\(
	1\le j\le n-\lfloor n^{3/5}\rfloor,
	\)
	\[
	M_{j,j}(x)
	=
	\frac{q+1}{m^{q+1}}
	(n-j+1)^q
	\Bigl(1-\frac{x}{q}\Bigr)^{q+1/2}
	\bigl(1+O(\lambda_n)\bigr),
	\qquad x<q.
	\]
	By \eqref{mjjll1}, the omitted diagonal tail is $o(\lambda_n)$. Moreover,
	\[
	\frac{q+1}{m^{q+1}}
	\sum_{j=1}^n(n-j+1)^q
	=
	1+O(m^{-1})
	=
	1+o(\lambda_n).
	\]
	Therefore
	$$	\operatorname{Tr}M(x)
	=
	\Bigl(1-\frac{x}{q}\Bigr)^{q+1/2}
	+O(\lambda_n),
	\qquad x<q. $$

	It remains to check that the determinant--trace approximation and
	the change of normalization do not affect the preceding error scales.
	Refining the Hilbert--Schmidt estimates used in Sections~3--5 shows
	that
	\[
	\|M(x)\|_{\mathrm{HS}}^2
	\]
	can be controlled at or below the corresponding trace-error scale in
	each regime. Since \(\|M(x)\|_{\mathrm{HS}}=o(1)\),
	\[
	\log\det({\rm I}-M(x))
	=
	-\operatorname{Tr}M(x)
	+O\bigl(\|M(x)\|_{\mathrm{HS}}^2\bigr),
	\]
	so the determinant--trace remainder does not change the stated orders.
	
	Finally, passing from \(\widehat X_n\) to \(X_n\) shifts the
	corresponding squared threshold by
	\[
	\frac{B_nx^2}{4A_n}
	=
	\frac{\lambda_nx^2}{4m(1-\lambda_n)},
	\]
	which is negligible compared with the relevant error scale in all
	three regimes. The bound $\mathbb{P}(R_n<0)\le e^{-n/2}$ from Section 2.1 is exponentially smaller than all three error scales. Hence the trace estimates
	above yield the stated pointwise error bounds.
\begin{rmk}[The three regimes and their proofs]
	The three regimes reflect the ranges in which different estimates for
	the beta tails are effective.
		
	For \(q\to\infty\), we further distinguish two ranges. When
	\(
	m^{2/5}(\log m)^2\lesssim q<m,
	\)
	Section~3 uses the gamma representation of the beta distribution
	together with a Cram\'er-type moderate deviation estimate to obtain a
	uniform Gaussian tail approximation. When
	\(
	1\ll q\ll m^{2/5}(\log m)^2,
	\)
	this uniform error control is no longer available over the full range
	needed in the proof, and Section~4 instead estimates the beta integrals
	directly. The latter method also requires \(q\) to remain sufficiently
	small relative to \(m\).
	
	When \(q\) is fixed, the Gaussian-type approximation used for
	\(q\to\infty\) is not available in the form needed here, so Section~5
	treats this case separately by an endpoint expansion.

	Thus the boundary
	\(
	q\asymp m^{2/5}(\log m)^2
	\)
	is  merely a technical division between two proof methods, rather than a
	transition in the limiting law. Its precise location is not intrinsic
	and is not unique.
\end{rmk}

\end{document}